\documentclass[12pt]{article}
\usepackage{setspace}
\usepackage{amsmath}
\usepackage{amsfonts}
\usepackage{amssymb}
\usepackage{amsthm}
\usepackage{mathrsfs}
\usepackage{amsxtra}
\usepackage[all,cmtip]{xy}
\usepackage{graphicx}
\usepackage{color}
\usepackage{comment}
\usepackage{verbatim}
\usepackage{hyperref}

\newtheorem{thm}{Theorem}[section]

\newtheorem{lem}[thm]{Lemma}
\newtheorem{prop}[thm]{Proposition}
\newtheorem*{conj}{Conjecture}
\newtheorem{cor}[thm]{Corollary}
\newtheorem*{thmA}{Theorem A}
\newtheorem*{thmB}{Theorem B}
\newtheorem*{thmC}{Theorem C}

\theoremstyle{definition}
\newtheorem{defn}[thm]{Definition}
\newtheorem{rem}[thm]{Remark}

\def\XXint#1#2#3{{\setbox0=\hbox{$#1{#2#3}{\int}$}
\vcenter{\hbox{$#2#3$}}\kern-.5\wd0}}

\def\bA{\mathbb{A}}
\def\cA{{\mathcal A}}

\def\cB{{\mathcal B}}

\def\cD{{\mathcal D}}

\def\cR{{\mathcal R}}

\def\cF{{\mathcal F}}

\def\cS{{\mathcal S}}

\def\cP{\mathcal{P}}

\def\cX{{\mathcal X}}
\def\Z{{\Bbb Z}}
\def\R{{\Bbb R}}

\def\C{{\Bbb C}}

\def\Q{{\Bbb Q}}

\def\P{\mathbb{P}}

\def\Hom{\text{Hom}}

\def\det{\text{det}\,}

\def\GL{\operatorname{GL}}
\def\SL{\operatorname{SL}}

\def\Symb{\operatorname{Symb}}

\def\GL{\operatorname{GL}}

\def\sign{\operatorname{sign}}
\def\Res{\operatorname{Res}}

\def\P{\mathbb{P}}
\def\cP{\mathcal{P}}

\def\LP{\mathcal{LP}}

\def\Symm{\operatorname{Sym}}

\def\sgn{\operatorname{sgn}}

\def\cH{\mathcal{H}}
\def\cZ{\mathcal{Z}}
\def\bD{\mathbf{D}}

\def\wt{\operatorname{wt}}
\def\ev{\operatorname{ev}}

\begin{document}
%\hoffset=-.4in
%\voffset=-0.05in

%\small

 \title{The $p$-adic Shintani modular symbol and evil Eisenstein series}
\author{G. Ander Steele}

\maketitle

\tableofcontents

\section{Introduction}
%Fix $p>2$ prime and et $f$ be a modular form of weight $k+2$ and level $\Gamma_1(N)$ prime to $p$.

Let $p>2$ be prime. In this paper we explicitly compute the $p$-adic $L$-function of critical slope, or evil, Eisenstein series, giving a new proof of the recent results of Bella\"iche \& Dasgupta \cite{BeDa}. The difficulty in computing the $p$-adic $L$-function of a critical slope Eisenstein series comes from the subtlety of its definition. Let us suppose $f=\sum_{n\geq -}q^n\in M_{k+2}(\Gamma_1(N),\varepsilon)$ is a normalized eigenform with level $N$ prime to $p$.  We would like to attach $p$-adic $L$-functions to $f$'s two $p$-refinements, $f_\alpha$, $f_\beta \in M_{k+2}(\Gamma_1(N)\cap\Gamma_0(p))$, where $\alpha,\beta$ are roots of the polynomial $x^2+a_px+\varepsilon(p)p^{k+1}$. Briefly, $f_\alpha$ and $f_\beta$ are the newforms on which $U_p$ acts by $\alpha,\beta$ and with tame Hecke eigenvalues the same as $f$. Note that the $p$-adic valuation of the roots (their \emph{slope}) must belong to the interval $[0,k+1]$. In the case that $v_p(\alpha)<k+1$,  Visik \cite{Vis} and Amice-Velu \cite{AmVe} have shown that there exists a unique $p$-adic analytic function interpolating the critical values of $L(f_\alpha\chi,s)$ with growth rate prescribed by the slope. If the slope of $\beta$ is equal to $k+1$, the growth rate of the desired $p$-adic $L$-function is too high for the function to be uniquely determined by the critical values of $L(f_\beta,\chi,s)$. Instead, one can use Steven's notion of \emph{overconvergent modular symbols} to attach a $p$-adic $L$-function to $f_\beta$.

Denote by $\cD_k$ the module of locally analytic distributions on $\Z_p$ equipped with a  weight-$k$ action of $\Gamma_0(p)$. For each congruence subgroup $\Gamma\subset \Gamma_0(p)$, the module of weight-$k$ $\Gamma$ overconvergent modular symbols, $\Symb_{\Gamma}(\cD_k)$ is defined as the space of $\Gamma$-invariant homomorphisms from degree $0$ divisors of $\P^1(\Q)$ to $\cD_k$. The module $\Symb_{\Gamma}(\cD_k)$ comes equipped with an action of Hecke and an involution $\iota$ decomposing $\Symb_{\Gamma}(\cD_k)=\Symb_{\Gamma}^+(\cD_k)\oplus\Symb_{\Gamma}^-(\cD_k)$. Stevens' control theorem implies that the $f_\alpha$-isotypical subspaces $\Symb^{\pm}_{\Gamma}(\cD_k)[f_\alpha]$ are one-dimensional whenever $v_p(\alpha)<k+1$ and $f$ a newform; if $f$ is an ordinary Eisenstein series, then one of the eigenspaces will be zero. Picking generators $\Phi_{f'_\alpha}^\pm$, Pollack-Stevens \cite{PolSte11} \emph{define} the $p$-adic $L$-function of $f_{\alpha}$ to be the distribution $\mu_{f_{\alpha}}=\Phi_{f_\alpha}^+([\infty]-[0])+\Phi_{f_\alpha}^-([\infty]-[0])$, and show that it has the correct interpolation and growth properties. We will adopt the conventions of \cite{Bel12} and \cite{BeDa}, writing
\begin{defn}
\begin{equation*}	L_p(f_{\alpha},s) := \int_{\Z_p^\times} z^s d\mu_{f_\alpha}
\end{equation*}
for all $s\in\Hom_{cts}(\Z_p^\times,\C_p^\times)$.
\end{defn}
If $v_p(\beta)=k+1$ but $f_\beta$ is not in the image of the operator $\theta^{k+1}=(q\frac{d}{dq})^{k+1}:M_{-k}^\dagger(\Gamma)\longrightarrow\ M_{k+2}^\dagger(\Gamma)$, Pollack-Stevens \cite{PolSte} show that the generalized eigenspaces $\Symb_{\Gamma}^\pm(\cD_k)(f_\beta)$ are one-dimensional \cite{PolSte} and \emph{define} the $p$-adic $L$-function of $f_\beta$ to be the Mazur-Mellin transform of $\mu_{f_\beta}$. Their definition does not apply to the case of evil Eisenstein series: the critical slope refinement of a weight $k+2$ Eisenstein series is $\theta^{k+1}$ of an ordinary family of (overconvergent) Eisenstein series at weight $-k$. Bella\"iche \cite{Bel12} has shown that the Hecke eigenspaces $\Symb_{\Gamma}^\pm(\cD_k)[f_\beta]$ are one-dimensional (assuming mild technical hypotheses and standard conjectures) even when $f_\beta$ is $\theta$-critical.

Before \cite{Bel12}, Pasol and Stevens used Pollack's MAGMA programs to compute an approximation to an overconvergent eigensymbol $\Phi\in\Symb^+_{\Gamma_0(p\ell)}(\cD_0)$ (for $p=3,\ell=11$) with the same Hecke eigenvalues as the critical slope refinement of
\begin{equation*}
	E_{2,\ell} = \frac{\ell-1}{24}+\sum_{n\geq 1}a_nq^n \in M_2(\Gamma_0(\ell)),\text{ where } a_n=\sum_{\substack{ d|n\\ \ell\nmid d}} d.
\end{equation*}
Based on their numerical experiments, they conjectured that the $E_{2,\ell}$-eigenspace of $\Symb^+_{\Gamma_0(p\ell)}(\cD_0)$ is $1$-dimensional, generated by $\Phi$,  and that the $p$-adic $L$-function of $\Phi$ is given by
\begin{conj}[Pasol-Stevens]
\begin{equation*}
L_p(\Phi,s) = \wt(s)  (1- \ell^{s}) \zeta_p(s+1)  \zeta_p(1-s).
\end{equation*}
\end{conj}

Bella\"iche's results prove, in particular, that $\Symb^+_{\Gamma_0(p\ell)}(\cD_0(p))[E_{2,\ell}]$ is one-dimensional, but they do not describe a generator of the eigenspace. The goal of this paper is to write down explicit generators of the critical slope Eisenstein eigenspaces of $\Symb_\Gamma^\pm(\cD_k)$, thereby computing the corresponding $p$-adic $L$-functions. We fall slightly short of our goal, giving instead explicit elements of $\Symb_{\Gamma}^\pm(\cD_k)$ which have the same value on the divisor $[\infty]-[0]$ as an eigensymbol. Of course, this is good enough for our purposes, since the $p$-adic $L$-functions of our modular symbols and the Eisenstein modular symbols will be the same. Our method is based on a careful study of the {Shintani modular symbol}, a universal Eisenstein modular symbol for $\GL_2^+(\Q)$ valued in a module of distributions on the finite adeles $\bA_{\Q^2}^{\infty}$. Building on our earlier work in \cite{Ste1}, we show how to interpret the Shintani modular symbol in terms of Stevens' $p$-adic distributions with rational poles, $\widetilde{\cD}(\Z_p^2)$. The main theorem of \S2 is
\begin{thmA}
Fixing a Bruhat-Schwartz function $f'\in\cS(\bA_{\Q^2}^{\infty,p})$ and $\Gamma_{f'}\subset\GL_2(\Q)$ stabilizing $f'$, the Shintani modular symbol $\Phi_{f'}$ is an element of $\Symb_{\Gamma_{f'}}(\widetilde{\cD}(\Z_p^2)/\delta_0)$.
\end{thmA}
\begin{rem}
Given a cuspform $f\in S_{k+2}(\Gamma)$, one forms a classical modular symbol in $\Symb_\Gamma(\Symm^k\C^2)$ by integrating the differential form $f(z)(zX+Y)^kdz$ against arcs between cusps, giving a polynomial in $\C[X,Y]_{k}\cong \Symm^k\C^2$. This recipe fails for Eisenstein series as the integrals fail to converge at cusps where the Eisenstein series has non-zero constant term. In \cite{Ste89}, Stevens constructs analogues of modular symbols valued in $\Q(X,Y)$ by integrating Eisenstein series on the Borel-Serre compactification of the upper half-plane. The values are polynomial terms, from integrating Eisenstein series along arcs in the upper half-plane, plus rational functions coming from integrals over ``modular caps" at the cusps. In some sense, the constant terms of Eisenstein series contribute the denominators of the rational functions and one can view $\widetilde{\cD}(\Z_p^2)$ as an enlarging of $\cD(\Z_p^2)$ that allows for the denominators.
\end{rem}

In \S3, we build specialization maps $\Symb_{\Gamma}\widetilde{\cD}(\Z_p^2)\longrightarrow \Symb_{\Gamma}\cD_k$ which send $\Phi_{f'}$ to Eisenstein overconvergent modular symbols. Recall that the ordinary specialization $g_{k,\alpha}$ of an Eisenstein series $g_k\in M_{k+2}(\Gamma_1(N))$ belongs to an analytic family of ordinary Eisenstein series, and this family can be specialized to $g_{-k-2,\alpha}\in M^\dagger_{-k}(\Gamma)$. Applying $\theta^{k+1}$ to $g_{-k-2,\alpha}$, one gets an evil Eisenstein series $f_\beta=\theta^{k+1}g_{-k-2,\alpha}\in M_{k+2}(\Gamma)$. One might hope to mimic this construction to produce critical slope Eisenstein modular symbols and, indeed, Bella\"iche has employed this strategy to compute the $p$-adic $L$-functions of critical slope CM cuspforms \cite{BelCM}. Unfortunately, this strategy only works for producing one of the two generators of the two eigenspaces $\Symb_{\Gamma}^+(\cD_k)[f_\beta]$ and $\Symb_\Gamma^-(\cD_k)[f_\beta]$, and this generator has trivial $p$-adic $L$-function. It should not be surprising that we can't specialize our symbols $\Phi_{f'}\in\Symb_{\Gamma}(\widetilde{\cD}(\Z_p^2))$ to ordinary symbols in $\Symb_{\Gamma}(\cD_{-2-k})$, but it is instructive to imagine how such an argument might go. In the absence of poles, one can define a $\Sigma_0(p)$-equivariant homomorphism $\rho_{-2-k}:\cD(\Z_p^2)\longrightarrow \cD_{-2-k}(\Z_p)$ by
\begin{equation}
	\int_{\Z_p} g(z) d\rho_{-2-k}(\mu)(z) = \int_{\Z_p\times\Z_p^\times} y^{-2-k} g(-x/y) d\mu(x,y),
\end{equation}
and composing with $\theta^{k+1}$  gives a $\Sigma_0(p)$-equivariant homomorphism $\theta^{k+1}\rho_{-2-k}:\cD(\Z_p^2)\longrightarrow \cD_k(\Z_p)(\det^{-k-1})$. Now $\rho_{-2-k}$ can not be extended to $\widetilde{\cD}$, but we manage to extend the homomorphism $\theta^{k+1}\circ\rho_{-2-k} \cD(\Z_p^2)\longrightarrow\cD_k(\Z_p)$ to a $\Gamma_0(p)$-equivariant homomorphism $\widetilde{\cD}(\Z_p^2)\longrightarrow \cD_k(\Z_p)$ by reinterpreting distributions as differential forms (see \cite{Sch84}). For a non-negative integer $k$ and an honest distribution $\mu\in \cD(\Z_p^2)$, $\theta^{k+1}\rho_{-2-k}\mu$ is the distribution 
\begin{equation*}
\int_{\Z_p} g(z) \theta^{k+1}\rho_{-2-k}\mu(z) =\int_{\Z_p\times\Z_p^\times} y^{-2-k}g^{(k+1)}(-x/y) d\mu(x,y)
\end{equation*}
for all $g\in\cA(\Z_p)$.The term $y^{-2-k}g^{(k+1)}(-x/y)$ is equal to the total residue over the disk $B[\Z_p,1]$ of the differential form $(k+1)!\frac{g(z)dz}{(x+yz)^{k+2}}$ for fixed $x,y\in\Z_p\times\Z_p^\times$. In fact, one can show 
\begin{align*}
	\int_{\Z_p} g(z) \theta^{k+1}\rho_{-2-k}\mu(z)
	=\int_{\Z_p\times\Z_p^\times} (k+1)!\Res_{B[\Z_p,1]} \frac{g(z)dz}{(x+zy)^{k+2}} d\mu(x,y)\\
	=(k+1)!\Res_{B[\Z_p,1]}\left(g(z)\cdot \int_{\Z_p\times\Z_p^\times}\frac{d\mu(x,y)}{(x+yz)^{k+2}}  dz\right),
\end{align*}
so that the  $\theta^{k+1}\rho_{-2-k}(\mu)$ is completely determined by the differential form $\int_{\Z_p\times\Z_p^\times}\frac{d\mu(x,y)}{(x+yz)^{k+2}}  dz$. Our strategy in \S3 is to build differential forms out of pseudo-distributions and view them, by the residue map, as distributions on $\Z_p$. For each positive integer $k$, it will be easy to specialize psuedo-distributions in $\widetilde{\cD}(\Z_p^2)$ to differential forms dual to $\cA_{-k}(\Z_p)$. We will show that  these differential forms vary in analytic families and we will produce forms dual to $\cA_k(\Z_p)$ by specializing the families to weight $-k$.  We show that the corresponding modular symbols $\Phi_{f'}^{-k}$ are Eisenstein symbols in $\Symb_{\Gamma}(\cD_k)$ with the desired sign, and Bellaiche's results on the dimension of the critical Eisenstein subspaces of $\Symb_{\Gamma}^\pm(\cD_k(\Z_p))$ allow us to show, for good choices of $f'$
\begin{thmB}
The $p$-adic $L$-function $L_p(\Phi_{f'}^{-k},s)$ is equal to the $p$-adic $L$-function of a critical slope Eisenstein series of weight $k+2$.
\end{thmB}
We use the explicit nature of the Shintani modular symbol to compute $L_p(\Phi_{f'}^{-k},s)$, which factors as a product of Kubota-Leopoldt $p$-adic $L$-functions, plus an extra factor vanishing at interpolation characters. Before stating the main theorem, let us set some notation. Denote by $\cX=\Hom_{cts}(\Z_p^\times,\C_p^\times)$, the $\C_p$-points of weight-space. For a character $s\in\cX$, we will write $\wt(s):=\log_p((1+p)^s)/\log_p(1+p)$ and $\sgn(s)=(-1)^s$. Note that the integers embed in $\cX$ by $k\mapsto (z\mapsto z^k)$.
\begin{thmC}
Let $k\geq 0$ be an even integer and $\psi,\tau$ primitive Dirchlet characters with conductors $M,L$ both prime to $p$. 
\begin{enumerate}
\item If $k=0$ and $(M,L)=1$ or $k>0$ and $\psi\neq 1$, we have
$$L_p(E_{k+2,\psi,\tau}^{crit},s)=L^{-s}{\wt{s}  \choose k+1}L_p(\tau^{-1},s-k-1)L_p(\psi,1-s)$$
for all $s\in\cX$ with $\sgn(s)=\psi(-1)$. 
\item The $p$-adic $L$-function of $E_{2,\ell}^{crit}$ is equal to
$$L_p(E_{2,\ell}^{crit},s)=\wt(s)(1- \ell^{s})\zeta_p(s+1)  \zeta_p(1-s)$$
for all $s$ with $\sgn(s)=1$.
\end{enumerate}
\end{thmC}

\begin{rem}
The $p$-adic $L$-function is vanishes for characters with the wrong sign--see \S1 of \cite{BeDa}.
\end{rem}
This strategy of using analytic families of ``Eisenstein" cocycles or modular symbols was suggested by the earlier work of Duff Campbell, who showed certain values of a similar cocycle vary in $p$-adic analytic families \cite{Duff} and Kalin Kostadinov \cite{Kalin}, who used Steven's notion of \emph{distributions with rational poles} to build an analytic family of modular symbols specializing at $0$ to a symbol with the same $p$-adic $L$-function predicted by Pasol and Stevens's conjecture.

Finally, we remark that our approach to Theorem C differs greatly from the recent work Bella\"iche and Dasgupta \cite{BeDa}. Their approach uses families of Eisenstein partial modular symbols--homomorphisms from divisors of $\P^1(\Q)$ supported at a subset of cusps $C$-- to build ordinary partial modular symbols $\Phi^{ord}\in\Symb_{\Gamma,C}(\cD_{-2-k})$ corresponding the the ordinary Eisenstein series $g_{\alpha,-k}$. Applying $\theta^{k+1}$ to $\Phi^{ord}$, they get a partial modular symbol $\theta^{k+1}\Phi^{ord}\in\Symb_{\Gamma,C}(\cD_k)$ with the eigenvalues of the evil Eisenstein series $f_{\beta}\in M_{k+2}(\Gamma)$. A comparison of partial modular symbols with full modular symbols shows that $\theta^{k+1}\Phi^{ord}$ is the restriction of the full modular symbol generating the evil eigenspace in $\Symb_{\Gamma}(\cD_k)$. Their explicit knowledge of $L_p(\theta_k\Phi^{ord},s)$ allows them to compute the desired $p$-adic $L$-function. 

\subsection{Acknowledgements}
The contents of this paper have been adapted from portions of my PhD thesis, which was supervised by Glenn Stevens; it is a great pleasure to acknowledge and give thanks for his enormous influence. This paper would not have been possible without our many discussions or his constant encouragement. Finally, I would like to thank Matthew Greenberg, Robert Pollack, and Jay Pottharst for many helpful conversations.

\subsection{Notation \& Conventions}
Let $V$ denote $\Q^2$, viewed as column vectors, so that $\GL_2(\Q)$ acts on $V$ via left multiplication. We will write $\cS(V)=\cS(\bA_{V}^{(\infty)})$ for the group of test functions on the finite adeles of $V$. Concretely, $\cS(V)$ is the group of functions $f:V\rightarrow\Z$ that are supported on a lattice and are also periodic with respect to a lattice. The group $\GL_2(\Q)$ naturally acts on $\cS(V)$ on the right by $(f|\gamma)(v)=f(\gamma v)$. Fixing a prime $p>2$, we denote by $\cS(V^{(p)})=\cS(\bA_{V}^{(\infty,p)})$ the group of test functions away from $p$, and $\cS(V_p)$ the group of test functions at $p$. Given $f'\in\cS(V^{(p)})$, we will write $\Gamma_{f'}\subset\GL_2(\Q)$ for the subgroup stabilizing $f'$; we will denote by $\Gamma_p\subset\GL_2(\Q)$ the subgroup of matrices stabilizing $\Z_p\times\Z_p^\times\subset\Q_p^2$. When no confusion is likely to arise, we will simply write $\Gamma$ to denote the intersection of $\Gamma_{f'}$ and $\Gamma_p$. 

Denote by $\cR$ the ring of power series $\Q[[X,Y]]$, a left $\Q[GL_2(\Q)]$-module under the action $\gamma\cdot F(X,Y)= F((X,Y)\gamma)=F(aX+cY,bX+dY)$ for all $\gamma=\begin{pmatrix} a & b\\ c& d\end{pmatrix}\in\GL_2(\Q)$. Let $S$ be the multiplicative subset of $\Q[X,Y]$ generated by non-zero linear polynomials $S=\langle aX+bY\neq 0\rangle$. Write $\widetilde{\cR}$ for the localization $S^{-1}\cR$, which is just the subset of $\widetilde{\cR}$ of power series with homogeneous denominators.

\section{The Shintani modular symbol}

\subsection{Shintani cocycles}
Solomon's Shintani cocycle (\cite{Sol98} and \cite{Sol99}) is a $1$-cocycle on $\GL_2(\Q)$ valued in a module of distributions inspired by Shintani's study of $L$-functions of totally real fields. These distributions are simply homomorphisms from the group of test functions $\cS(V)$ to $\widetilde{\cR}$. Define
\begin{equation*}
	\cD(V,\widetilde{\cR}):=\Hom_{\Z}(\cS(\bA_{\Q^2}^{\infty}), \widetilde{\cR} ),
\end{equation*}
and equip it with two actions of $\GL_2(\Q)$. The left action of $\GL_2$ on $V$ induces a right action on $\cS(V)$, and a left action on $\cD(V,\widetilde{\cR})$ by putting, for each $f\in\cS(\bA_{V}^{\infty})$ and $\gamma\in\GL_2(\Q)$, $(\gamma\cdot \mu)(f) = \gamma \cdot \mu(f|\gamma).$

\begin{rem}
We endow $\cD(V,\widetilde{\cR})$ with a right $\GL_2(\Q)$ by $\mu|\gamma:=\gamma^*\cdot\mu$, where $\gamma^*=\det(\gamma)\gamma^{-1}$. We use the adjugate action because it preserves integrality of matrices, which will be useful to use when we are computing Hecke operators.
\end{rem}

The key idea is the construction of distributions on $\Q^2$ from cones. Briefly, a cone $C\subset\R_2$ (say the positive span of two rational vectors $v,w\in\R^2$, $C=C(v,w)$) determines a distribution $\mu_C\in\cD(V,\widetilde{\cR})$ by sending a test function $f$ to the power series representing
\begin{equation}\label{E:mu_def}
	\sideset{}{'}\sum_{u\in \Q^2\cap C} f(u)e^{u\cdot(X,Y)}.
\end{equation}
Here $(v_1,v_2)\cdot(X,Y)=v_1X+v_2Y$ and we are weighting points on the edge of the cone by $1/2$ and the origin by $0$. Viewing the formal sum as a function of $X,Y$, we see that it converges for $X,Y$ in the dual cone of $C$. After rescaling by positive rational numbers, we may assume $f$ is periodic with respect to the vectors $v,w$ and the sum converges to
\begin{equation*}
	\frac{1}{1-e^{v\cdot(X,Y)}}\frac{1}{1-e^{w\cdot(X,Y)}} \sideset{}{'}\sum_{v\in \Q^2 \cap P} f(v) e^{v\cdot(X,Y)},
\end{equation*}
where $P$ is the parallelogram $\{ \alpha v+\beta w: 0\leq \alpha,\beta\leq 1\}$. Observe that $(aX+bY)\cdot\frac{1}{1-e^{aX+bY}}=\sum_{n\geq} (-1)^nB_n\frac{(aX+bY)^n}{n!}\in\Q[[X,Y]]$ and thus $\frac{1}{1-e^{aX+bY}}$ has a well-defined power series in $\Q((X,Y))$ with homogenous denominator $aX+bY$. Therefore, we may view the formal sum (\ref{E:mu_def}) as representing an element of $\widetilde{\cR}$ which we define to be $\mu_C(f)$.

Solomon defines an homogenous $1$-cocycle on $\GL_2^+(\Q)$ by sending a tuple $(\alpha,\beta)\in\GL_2(\Q)^2$ to the distribution $\det(\alpha e_1,\beta e_1) \mu_{C(\alpha e_1,\beta e_1)}$.
\begin{thm}[Solomon]
There exists a $\GL_2^+(\Q)$ cocycle $\Psi \in Z^1(\GL_2^+(\Q),\cD(V,\widetilde{\cR})/\delta_0)$ satisfying
\begin{enumerate}
\item If $\sigma\in\GL_2^+(\Q)$ sends the cusp $\infty$ to $0$, then $\Psi( 1,\sigma) (f) = \sideset{}{'}\sum_{v\in \Q_+^2} f(v) e^{v\cdot(X,Y)}$ for all $f\in\cS(\Q^2)$. 
\item For all $\lambda\in\Q^\times$, $\alpha,\beta\in\GL_2^+(\Q)$, $\Psi(\lambda \alpha,\lambda\beta)=\Psi(\alpha,\beta)$.
\item $\Psi$ is parabolic: For each cusp $r\in\P^1(\Q)$, denote by $\Gamma_v\subset\SL_2(\Z)$ the stabilizer of $r$. The restriction of $\Psi$ to $\Gamma_r$ is trivial for all cusps $r\in\P^1(\Q)$.
\end{enumerate}
\end{thm}

\subsection{Modular symbols}
Let $\Gamma$ be a subgroup of $\GL_2(\Q)$, which acts on $\P^1(\Q)$ via fractional linear transformations: $\gamma\cdot s=\frac{as+b}{cs+d}.$ Let $\Delta_0$ denote the group of degree $0$ divisors on $\P^1(\Q)$ with the induced $\Gamma$-action. For any two cusps $r,s\in\P^1(\Q)$, our convention will be to write $\{r,s\}$ for the divisor $[s]-[r]\in\Delta_0$. Given a right $\Z[\Gamma]$-module $M$, $\Gamma$ acts on $\Hom_{\Z}(\Delta_0,M)$ by $(\varphi|\gamma)(D):=\varphi(\gamma D) |\gamma.$  A $M$-valued modular symbol is a $\Gamma$-invariant homomorphism $\varphi:\Delta_0\longrightarrow M$. That is, $\varphi(\gamma D)|\gamma=\varphi(D)$.
\begin{defn}
The module of $M$-valued modular symbols is denoted
\begin{equation*}
\Symb_{\Gamma}(M):=\Hom_{\Z}(\Delta_0,M)^\Gamma.
\end{equation*}
\end{defn}

Properties (2) and (3) of Solomon's cocycle allow us to interpret it as a $\GL_2^+(\Q)$ modular symbol, viewing $\cD(V,\widetilde{\cR})$ as a right $\GL_2^+(\Q)$-module by $\mu|\gamma:=\gamma^*\cdot\mu$.

\begin{thm}
There exists a $\GL_2^+(\Q)$ modular symbol $\Psi\in\Symb_{\GL_2^+(\Q)}(\cD(V,\widetilde{\cR})/\delta_0)$ characterized by
\begin{equation*}
	\Psi\{0,\infty\} (f)=\sideset{}{'}\sum_{v\in \Q_+^2} f(v) e^{v\cdot(X,Y)}
\end{equation*}
for all $f\in\cS(V)$.
\end{thm}
\begin{proof}
For any cusps $r,s,\in \P^1(\Q)$, let $\gamma_r,\gamma_s$ be choices of matrices sending $\infty$ to $r$ and $s$, respectively. We can define $\Psi\{r,s\}:=\Psi(\gamma_r,\gamma_s)$, and by the parabolic property of $\Psi$ this is well-defined. The cocycle property of $\Psi$ implies that the map $\{r,s\}\mapsto \Psi\{r,s\}$ extends to a homomorphism $\Delta_0\rightarrow \cD(V,\widetilde{\cR})/\delta_0$. It remains to check the $\GL_2(\Q)$-invariance of $\Psi$: For any $\alpha\in\GL_2^+(\Q)$ and cusps $r,s\in\P^1(\Q)$, $\Psi\{\alpha r,\alpha s\}|\alpha = \alpha^*\Psi(\alpha\gamma_r,\alpha\gamma_s)=\det(\alpha) \alpha^{-1}\cdot\Psi(\alpha\gamma_r,\alpha\gamma_s)=\det(\alpha)\cdot \Psi(\gamma_r,\gamma_s)$. The invariance of $\Psi(\gamma_r,\gamma_s)$ under scalars shows that this is equal to $\Psi\{r,s\}$, as desired.
\end{proof}

We record two useful lemma's about the modular symbol $\Psi$. First, observe that the denominator of $\Psi\{\infty,0\}(f)$ is $XY$, i.e. $XY\Psi\{\infty,0\}(f)\in\Q[[X,Y]]$ for all $f\in\cS(V)$. Using the $\GL_2(\Q)^+$-invariance of $\Psi$ and the fact that $\GL_2(\Q)^+$ acts transitively on divisors $\{r,s\}\in\Delta_0$, we have
\begin{lem}\label{L:pole_desc}
For all $\frac{a}{c},\frac{b}{d}\in\P^1(\Q)$ and $f\in\cS(V)$,  $(aX+bY)(cX+dY)\Psi\left\{\frac{a}{c},\frac{b}{d}\right\}(f)\in\Q[[X,Y]].$\end{lem}

Before stating the next lemma, let us observe that some test functions on $\Q^2$ factorize (with respect to our chosen basis) as the product of test functions on $\Q$. For example, $[m\Z\times n\Z]$ is the product of the test functions $[m\Z]$ and $[n\Z]$: $[m\Z\times m\Z](xe_1+ye_2)=[m\Z](x) [n\Z](y)$. If $f_1,f_2\in\cS(\Q)$ are test functions, we will write $f_1\times f_2$ for the function
\begin{equation*}
	(f_1\times f_2) (x,y) := f_1(x) f_2(y).
\end{equation*}
Since $f_1,f_2$ are supported on latices in $\Q$, the product $f_1\times f_2$ is supported on a lattice in $\Q^2$. Moreover, if $f_1,f_2$ are periodic with respect to the lattice $m\Z$, $n\Z$, then $f_1\times f_2$ is periodic with respect to $m\Z\times n\Z$. If $f$ can be factored as the product of two test functions on $\cS(\Q)$, say $f(x,y)=f_1(x)f_2(y)$, then
\begin{align*}
	\Psi\{\infty,0\}(f_1\times f_2)
	=\left(\frac{1}{1-e^{mX}}\sideset{}{'}\sum_{0\leq x< m}f_1(x)e^{xX}\right)\left(\frac{1}{1-e^{nY}}\sideset{}{'}\sum_{0\leq y< n}f_2(y)e^{yY}\right),
\end{align*}
the prime on the sums indicating weighting the $x=0$ and $y=0$ terms by $1/2$. 

Every test function $f\in\cS(V)$ is a finite sum of factorizable test functions $f_1\otimes f_2$. Writing $\xi\in\cD(\Q,\Q((X)))$ for the distribution coming from the $1$-dimensional cone of positive real numbers, the above observations show
\begin{lem}
The distribution $\Psi\{\infty,0\}$ factors as a product of distributions
\begin{equation}
	\Psi\{\infty,0\} = \xi_1\times\xi_2,
\end{equation}
with $\xi_1$ valued in $\Q((X))$, $\xi_2$ valued in $\Q((Y))$, and $\xi_1\times\xi_2(f_1\times f_2)=\xi_1(f_1)\xi_2(f_2).$
\end{lem}
%Then $\Psi$ is characterized by
%\begin{equation}
%	\Psi\{\infty,0\} ([\bx+\Z^2]) = \frac{1}{1-e^X}\left(e^{\{x_0\}X}-\frac{1}{2}\delta_0\{x_0\}\right)\frac{1}{1-e^Y}\left(e^{\{x_1\}Y}-\frac{1}{2}\delta_0\{x_0\}\right),
%\end{equation}
%where $\{x\}$ denotes the fractional part of $x$.
 %through $\cK_{\Q^2}/\cL_{\Q^2}\longrightarrow\cD(\Q^2,\Q((x_1,x_2)))/\delta_0$.
%
% It follows that we have a $\GL(V)$-equivariant homomorphism $\cK_V/\cL_V\longrightarrow \cD(V,S^{-1}\Q[[V]])/\delta_0$. Applying this homomorphism to Solomon's Shintani cocycle, we 

\subsection{Locally polynomial distributions and the Fourier transform}
Viewing $x,y$ as coordinates on $V$, we regard $\cS(V)\otimes_\Z \Q[x,y]$ as a space of locally polynomial functions on $V$. We denote by
\begin{equation}
	\cD_{poly}(V,\Q):=\Hom_{\Q} (\cS(V)\otimes_\Z \Q[x,y],\Q)
\end{equation}
the space of locally polynomial distributions on $V$.For a locally polynomial function $f\otimes P$, we write write $\int_V f(v)P(v) d\mu(v):=\mu(f\otimes P)$.

Let us forget for a moment that $\Psi$ is valued in distributions valued in Laurent series and pretend that these distributions were valued in power series. We can identify the $\Q$-vector space $\Q[[X,Y]]$ with $\Q[x,y]^*$, the linear dual of the polynomial ring $\Q[x,y]$. This induces an isomorphism
\begin{align}
	%\cD_{poly}(V,\Q):=&
	\Hom_{\Q} (\cS(V)\otimes_\Z \Q[x,y],\Q)
	=& \Hom_{\Q}(\cS(V),\Q[x,y]^*)\\
	\cong& \Hom_{\Q}(\cS(V),\Q[[X,Y])
	=\cD(V,\Q[[X,Y]]).
\end{align}

The isomorphism $\cD_{poly}(V,\Q)\longrightarrow \cD(V,\Q[[X,Y]])$, which we call the ``Fourier transform" $\cF$, is explicitly described by
\begin{equation}
	\cF(\mu)(f) = \int_V f(x,y)\exp(xX+yY) d\mu(x,y):= \sum_{n,m\geq 0} \mu( f\otimes x^ny^m)\frac{X^nY^m}{n!m!}.
\end{equation}

We state the key properties of $\cF$. The proof is an easy computation which we omit.
\begin{prop}\label{P:Fourier}
The Fourier transform $\cF:\cD_{poly}(V,\Q)\longrightarrow\cD(V,\Q[[X,Y]])$ is an isomorphism of $\Q$-vector spaces. Moreover,
\begin{itemize}
\item[(i)] $\cF$ is $\GL_2(\Q)$-equivariant: $\cF(\mu|\gamma)=\cF(\mu)|\gamma$ for all $\gamma\in\GL_2(\Q)$.
\item[(ii)] $\cF(\partial_x\cdot\mu)=X\cF(\mu)$, $\cF(\partial_y \cdot \mu)=Y\cF(\mu)$
\item[(iii)] If $v$ has coordinates $a,b$, then $\cF(\delta_{v})=e^{aX+bY}\delta_{v}$.
\end{itemize} 
Here the action of $\partial_x, \partial_y$ on distributions is induced by the action on polynomials: $(\partial_x \cdot\mu)(f\otimes P):=\mu(f\otimes\partial_xP)$.
\end{prop}
The (inverse) Fourier transform realizes power-series-valued distributions as linear functionals on the locally polynomial functions of $V$. These, in turn, can be interpreted as linear functionals on $p$-adic polynomial functions of compact support. In the best of cases, these distributions are continuous, allowing one to obtain locally \emph{analytic} $p$-adic distributions. Of course, the values of the Shintani cocycle are distributions in $\cD(V,\widetilde{R})$,--the poles corresponding to the poles of Shintani zeta functions. Therefore, we must allow for poles in our locally polynomial distributions. This can be achieved using an ideas of Stevens, localizing $\cD_{poly}$ with respect to rational differential operators.

We let $\Q[X,Y]$ act on $\cD_{poly}$ by putting $P(X,Y)\cdot\mu:=P(\partial_x,\partial_y)\mu$. The space of \emph{locally polynomial distributions with rational poles} is defined as
\begin{equation}
	\widetilde{\cD}_{poly}(V):=S^{-1}\widetilde{\cD}_{poly}(V)=\cD_{poly}(V)\otimes_{\Q[X,Y]} S^{-1}\Q[X,Y],
\end{equation}
where again, $S$ is the multiplicative subset generated by non-zero linear forms.
\begin{lem}
The Fourier transform induces an isomorphism
\begin{equation}
	\cF:\Symb_{\GL_2^+}(\widetilde{\cD}_{poly}(V)/\delta_0)\longrightarrow \Symb_{\GL_2^+}(\cD(V,\widetilde{\cR})/\delta_0).
\end{equation}
\end{lem}
\begin{proof}
Proposition \ref{P:Fourier}, (ii) says the Fourier transform is a $\Q[X,Y]$-homomorphism, so it extends to an isomorphism of the localizations $\cF:\widetilde{\cD}_{poly}(V)\longrightarrow\cD(V,\widetilde{\cR})$. The third property shows $\cF(\delta_0)=\delta_0$, so we have an isomorphism on the quotient $\cF:\widetilde{\cD}_{poly}(V)/\delta_0\longrightarrow \cD(V,\widetilde{\cR})/\delta_0$, and the $\GL_2$-equivariance induces the isomorphism on modular symbols.
\end{proof}
It follows that we can realize the Shintani modular symbol as taking values in polynomial distributions with rational poles:
\begin{cor}\label{C:def}
There exists a unique modular symbol $\Phi\in\Symb_{\GL_2}(\widetilde{\cD}_{poly}(V)/\delta_0)$ such that $\cF(\Phi)=\Psi$.
\end{cor}

\subsection{Modules of $p$-adic distributions}
%We fix, for the remainder of this section, a lattice $L\subset V$ and a prime $p$. Completing $L$ at $p$, we get a compact open $L_p=L\otimes_\Z\Z_p\subset V_p$.
Fix a prime $p$, and let $U$ be a compact open subspace of $\Q_p$ or $\Q_p^2$. In this section, we recall the definitions of the distribution spaces we will use. Our main reference is \cite{PolSte11}.
\begin{itemize}
	\item $B[U,r]:=\{z\in W_p\widehat\otimes\C_p ~|~ \exists a\in U \text{ s.t.} |z-a|\leq r\}.$
	\item $A[U,r]:=$ the $\Q_p$-Banach algebra of rigid analytic functions on $B[U,r]$
		{whose taylor expansions on $U$ have $\Q_p$-coefficients.}
	\item $P[U,r]:=$ the subspace in $A[U,r]$ of polynomial functions.
	\item $D[U,r]:=${ the Banach dual of $A[U,r]$}.
\end{itemize}
For each $r>s>0$, restriction induces a completely continuous homomorphism $A[U,r]\longrightarrow A[U,s]$. The space of locally analytic functions on $U$ is naturally identified with the injective limit $\cA(U)=\varinjlim_{r>0} A[U,r]$ endowed with the inductive limit topology. We also define, for fixed $s$, the overconvergent analytic functions $\cA^\dagger(U,s)=\varinjlim_{r>s} A[U,r]$ and endow it with the inductive limit topology. Within the space of locally analytic functions, we have the dense subspace of \emph{locally polynomial} functions, $\LP(U)$, the union of the $P[U,r]$.

The locally analytic distributions, overconvergent distributions, and locally polynomial distributions are defined by
\begin{itemize}
\item $\cD(U):=\Hom_{cts}(\cA(U),\Q_p)=\varprojlim_{s>0}D(U,s)$
\item $\cD^\dagger(U,r):=\Hom_{cts}(\cA^\dagger(U,r),\Q_p)=\varprojlim_{s>r} D[U,r]$
\item $\cD_{poly}(U):=\Hom_{\Q_p}(\LP(U),\Q_p)$.
\end{itemize}
\begin{rem}
Following \cite{PolSte11}, abbreviate $\cD^\dagger(\Z_p,1)$ by $\cD^\dagger(\Z_p)$ or $\cD^\dagger$ and $D[\Z_p,1]$ by $\bD[\Z_p]$ or $\bD$.
\end{rem}

If $U\subset \Q_p^2$, the differential operators $\partial_x,\partial_y$ act on $\LP(U)$ and $\cA(U)$ in the usual way. We endow these spaces with a $\Q[X,Y]$-action by $P(X,Y)\cdot f(x,y):=P(\partial_x,\partial_y)f(x,y)$, thereby defining an action of $\Q[X,Y]$ on $\cD(U)$ and $\cD_{poly}(U)$. Just as in the case of global distributions, define the corresponding spaces of $p$-adic distributions with rational poles by
\begin{equation}
	\widetilde{\cD}(U):=S^{-1}\cD(U) \text{ and }\widetilde{\cD}_{poly}(U):=S^{-1}\cD_{poly}(U).
\end{equation}

\begin{rem}
Because we are viewing $\Q_p^2$ as column vectors, we have a left action of $\GL_2(\Q_p)$ on $\Q_p^2$, and thus a right action on function spaces and a left action on distribution modules. Again, we endow our distribution modules with a right action via $\mu|\gamma:=\gamma^*\cdot \mu$. If $\Gamma\subset\GL_2(\Q_p)$ has the property that $u|\gamma\in U$ for all $\gamma\in \Gamma$, $u\in U$, then  $\Gamma$ acts on $\cD(U)$. In particular, the semigroup $\Sigma_0(p)$ acts on $\cD(\Z_p^2)$ and $\cD(\Z_p\times\Z_p^\times)$.
\end{rem}

\subsection{$p$-adic distributions from global distributions}
Now we pass from global polynomial distributions to $p$-adic polynomial distributions on $\Z_p^2\subset V_p$. Put $\cS(\Z_p^2)\subset\cS(V_p)$ for the group of test functions supported on $\Z_p^2$. 
\begin{lem}\label{localpolyid}
The space of locally polynomial functions $\LP(\Z_p^2)$ is isomorphic to $\cS(\Z_p^2)\otimes_\Z\Q_p[x,y]$.
\end{lem}
\begin{proof}
The product map $f(x,y)\otimes P(x,y)\mapsto f(x,y)P(x,y)$ is injective and identifies $\cS(U)\otimes\Q[x,y]$ as a subset of $\LP(\Z_p^2)$. Moreover, every locally polynomial function $F$ on $\Z_p^2$ is the sum of finitely many polynomial functions $P_i$ restricted to compact open subsets of $V_i\subset \Z_p^2$, extended by zero. The sum $\sum [V_i]\otimes P_i$ realizes $F$ as an element of $\cS(U)\otimes\Q_p[x,y]$.
\end{proof}

A test function $f'\in\cS(V^{(p)})$ induces a homomorphism
\begin{align*}
	\LP(\Z_p^2)\cong\cS(\Z_p^2)\otimes\Q[x,y]& \longrightarrow\cS(V)\otimes\Q[x,y],\\
	f_p\otimes P& \mapsto (f_p\otimes f')\otimes P.
\end{align*}

Dually, we have a homomorphism $\cD_{poly}(V)\longrightarrow \cD_{poly}(\Z_p^2)$ which is concretely given by sending a distribution $\mu$ to the map $\mu_{f'}: P \mapsto \mu(f'\otimes[\Z_p^2]\otimes P)$.
% given a locally polynomial distribution $\mu\in\cD_{poly}(V,W)$, there is a unique linear map $\mu_{f'}:\LP(U)\longrightarrow \Q_p$ making the following diagram commute:
%\begin{eqnarray}
%\xymatrix{
%\LP(U) \ar[drr]_{\mu_{f'}} \ar[r]^-\sim & \cS(U)\otimes\Q_p[x,y] \ar[r]^-{\otimes f'} & \cS(V)\otimes\Q[x,y]\otimes_\Q\Q_p\ar[d]^{\mu}\\
%&  & \Q_p
%}
%\end{eqnarray}

Localizing with respect to $S$, we record the following lemma.
\begin{lem}\label{L:local_from_global}
For each $f'\in\cS(V^{(p)})$, there is a $\Q$-linear map 
\begin{equation}
\widetilde{\cD}_{poly}(V)\longrightarrow\widetilde{\cD}_{poly}(\Z_p^2).
\end{equation}
\end{lem}
We will write $\mu_{f'}$ for the image of $\mu$ under the map induced by $f'$.
%
%\begin{prop}
%There exists a modular symbol $\Psi\in\Symb_{\GL_2^+}(\cD(V,\cR)/\delta_0)$ characterized by
%\begin{equation*}
%	\Psi\{\infty,0\}(f) = \sum_{v\in \Q_+^2}{}' f(v) e^{v\cdot(X,Y)}
%\end{equation*}
%\end{prop}
%
%\begin{thm}
%There exists a modular symbol $\Phi\in\Symb_{\GL_2^+}(\Hom(\cS(V^{(p)}),\widetilde{\cD_p}/\delta_0))$ whose Fourier transform $\cF(\Phi)=\Psi$.
%\end{thm}
\subsection{The $p$-adic Shintani modular symbol}
Applying the previous sections to the values of modular symbol $\Phi$ gives $p$-adic polynomial distributions with rational poles. In this section, we show that the values of $\Phi$ extend to locally analytic distributions with rational poles.

\begin{thm}\label{T:is_analytic}
For all $f'\in\cS(V^{(p)})$, $\Phi\{\infty,0\}_{f'}\in\widetilde{\cD}(\Z_p^2)\subset\widetilde{\cD}_{poly}(\Z_p^2)$.
\end{thm}

The main idea behind the proof is the famous theorem of Amice and Velu. Before stating their theorem, we setup notation. Given an analytic distribution $\mu\in\cD(\Z_p^2)$, we can define a function $\cA(\mu)$ on the open polydisk $B(1,1)^2$ by
\begin{equation*}
	\cA(\mu)(q_1,q_2):=\int_{\Z_p^2} q_1^xq_2^y d\mu(x,y),
\end{equation*}
where we are using the fact that $x\mapsto q^x$ defines an analytic function on $\Z_p$. The theorem of Amice and Velu (extended to two variables) says that $\cA(\mu)$ is an analytic function on the polydisk $B(1,1)^2$ defined over $\Q_p$ and that all such analytic functions arise in this way.
 
\begin{thm}[Amice-Velu \cite{AmVe}]
The Amice transform $\cA:\cD(\Z_p^2)\longrightarrow \cA_\Q(B(1,1)^2)$ is an isomorphism of $\Q_p$-Frech\'et spaces.
\end{thm}

\begin{rem}We can use the $p$-adic exponential function to identify the disk $B(0,p^{-1/p-1})^2$ with $B(1,1)^2$ by $(X,Y)\mapsto (\exp_p(X),\exp_p(Y))$. Under this identification, the Amice transform becomes
\begin{equation*}
	\cA(\mu)(\exp_p(X,Y))=\int_{\Z_p^2} \exp_p(X)^x\exp_p(Y)^yd\mu(x,y)=\int_{\Z_p^2}e^{xX+yY}d\mu(x,y),
\end{equation*}
which, formally, is our map $\cF$ pushed forward to $\cD(\Z_p^2)$.
\end{rem}
\begin{rem}
Expanding $q_1^xq_2^y=\sum_{n,m\geq 0} {x \choose n}{y \choose m} (q_1-1)^n(q_2-1)^m$, we can extend $\cA$ to a homomorphism $\cD_{poly}(\Z_p^2)\longrightarrow\Q_p[[q_1-1,q_2-1]]$, where now $\cA(\mu)$ represents a formal power series which may or may not converge on $B(1,1)^2$.
\end{rem}
Armed with these observations, we are ready for the proof of Theorem \ref{T:is_analytic}.
\begin{proof}
Fix $f'\in\cS(V^{(p)})$ and let $\mu_{f'}=\Phi\{\infty,0\}_{f'}$. Since $XY\Psi\{\infty,0\}(f)\in\Q[[X,Y]]$, $XY\Phi\{\infty,0\}\in\cD_{poly}(V)$ and we conclude $XY\mu_{f'}\in\cD_{poly}(\Z_p^2)$. We compute the Amice transform of $XY\cdot \mu_{f'}$ and recognize it as an analytic function on $\cB(1,1)^2$, whence $XY\cdot\mu_{f'}\in\cD(\Z_p^2)$ and $\mu_{f'}\in\widetilde{\cD}(\Z_p^2)$. 
On the polydisk about $(0,0)$, the Amice transform is
\begin{align*}
	\cA(XY\mu_{f'})(\exp_p(X,Y))&=\int_{\Z_p^2}e^{xX+yY}d (XY\cdot\mu_{f'})\\
	&=\int_{V} f'\otimes[\Z_p^2]e^{xX+yY} d(XY\cdot\Phi\{\infty,0\})\\
	&=\cF(XY\Phi\{\infty,0\})(f'\otimes[\Z_p^2])\\
	&=XY\Psi\{\infty,0\}(f'\otimes[\Z_p^2]).
\end{align*}
Note that $f'\otimes[\Z_p^2]$ is periodic at $p$ with respect to $\Z_p^2$, so we may take as a period lattice of $f'\otimes[\Z_p^2]$ $m\Z\times n\Z$ with $m,n\in\Q^\times$ $p$-adic units. Using equation \ref{E:explicit} for $\Psi\{\infty,0\}(f'\otimes[\Z_p^2])$, we are left with
\begin{equation*}
	\cA(XY\mu_{f'})=\Psi\{\infty,0\}(f)=\frac{X}{1-e^{mX}}\frac{Y}{1-e^{nY}} \sum_{ (x,y)\in\cP}f'\otimes[\Z_p^2](x,y)e^{xX+yY}.
\end{equation*}
In the variables $q_1,q_2$, this becomes
\begin{equation*}
	\cA(XY\mu_{f'})(q_1,q_2)= \frac{\log(q_1)}{1-q_1^m}\frac{\log(q_2)}{1-q_2^n}\sum_{(x,y)\in\cP}f'\otimes[\Z_p^2](x,y)q_1^xq_2^y.
\end{equation*}

Now the sum is a finite number of non-zero terms, each with $x,y\in\Z_p$. Thus, each summand is an analytic function on the open polydisk $B(1,1)^2$. We are reduced to showing the that $\log(q_1)/(1-q_1^m)$ and $\log(q_2)/(1-q_2^n)$ define analytic functions on $B(1,1)^2$. We use crucially that $m,n$ are $p$-adic units, which implies $\frac{1-q_1^m}{1-q_1}$, $\frac{1-q_2^n}{1-q_2}$ are units on $B(1,1)^2$. 
\begin{equation*}
	\log(q)/(q-1)=\frac{1}{q^m-1}\sum_{n\geq 1} \frac{(-1)^n}{n} (q-1)^n\frac{(-1)^n}{n}=-\sum_{n\geq 0} \frac{(-1)^{n+1}}{n+1} (q-1)^n.
\end{equation*}
The above power series converges for all $q\in\C_p$ such that $|q-1|_p<1$, so $\log(q)/1-q$ is an analytic function on the open unit ball $B(1,1)$. Therefore $\log(q_1)/(1-q_1^m)$ and $\log(q_2)(1-q_2^n)$ are analytic functions on $B(1,1)^2$. Thus, $XY\cdot\mu_{f'}\in\cD(\Z_p^2)$ by the theorem of Amice-Velu.
\end{proof}

\begin{lem}\label{L:S0_action}
If $\gamma\in\GL_2(\Q)$ fixes $\Z_p^2$, then
\begin{equation*}
	\Phi_{f'}|\gamma=\Phi_{f'|\gamma^{*-1}}.
\end{equation*}	
\end{lem}
\begin{proof}
Fix $\gamma\in\GL_2(\Q)$, $D\in\Delta_0$, and $P\in\Q[X,Y]$. We apply the distribution $\Phi(\gamma\cdot D)_{f'}|\gamma$ to $P$:
\begin{align*}
	\Phi(\gamma\cdot D)_{f'}|\gamma(P)&=\Phi(\gamma\cdot D)_{f'} (P|\gamma^*)\\
							&=\Phi(\gamma\cdot D) (f'\otimes[\Z_p^2]\otimes P|\gamma^*)\\
							&=\Phi(\gamma\cdot D)( (f'|(\gamma^*)^{-1}\otimes[\Z_p^2]\otimes P)|\gamma^*)\\
							&=(\Phi(\gamma\cdot D)|\gamma)(f'|(\gamma^*)^{-1}\otimes[\Z_p^2]\otimes P)\\
							&=\Phi(D)(f'|(\gamma^*)^{-1}\otimes[\Z_p^2]\otimes P)\\
							&=\Phi(D)_{f'|(\gamma^*)^{-1}}(P).
\end{align*}
As $P$, $D$ were arbitrary, we conclude $\Phi_{f'}|\gamma = \Phi_{f'|\gamma^{*-1}}$.
\end{proof}

Restriction to $\SL_2(\Z)$, where $\gamma^*=\gamma^{-1}$, we record a corollary:
\begin{lem}\label{L:SL_action}
For all $\gamma\in \SL_2(\Z)$,
\begin{equation*}
	\Phi_{f'}|\gamma=\Phi(D)_{f'|\gamma}
\end{equation*}
\end{lem}
Now fix $f'\in\cS(V^{(p)})$, and let $\Gamma_{f'}\subset \SL_2(\Z)$ be the stabilizer of $f'\otimes[\Z_p^2]\in\cS(V)$. 
\begin{defn}
Define
\begin{equation*}
\Phi_{f'}:\Delta_0\longrightarrow \widetilde{\cD}_{poly}(\Z_p^2)/\delta_0
\end{equation*}
by $\Phi_{f'}(D):=\Phi(D)_{f'}$. By the above Lemma, $\Phi_{f'}|\gamma=\Phi_{f'}$ for all $\gamma\in \Gamma_{f'}$; therefore, $\Phi_{f'}$ is a $\Gamma_{f'}$ modular symbol valued in $\widetilde{\cD}_{poly}(\Z_p^2)/\delta_0$. In fact,
\end{defn}
\begin{thm}\label{T:analytic_modular_symbol}
For each $f'\in\cS(V^{(p)})$, $\Phi_{f'}$ is valued in analytic distributions with rational poles, $\widetilde{\cD}(\Z_p^2)/\delta_0$. That is, $\Phi_{f'}\in\Symb_{\Gamma_{f'}}(\widetilde{\cD}(\Z_p^2)/\delta_0)$.
\end{thm}
\begin{proof}
It suffices to show that $\Phi_{f'}(\gamma\{\infty,0\})\in\widetilde{\cD}(\Z_p^2)$ for all $\gamma\in\SL_2(\Z)$. Lemma \ref{L:SL_action} tells us $\Phi_{f'}(\gamma\{\infty,0\})|\gamma=\Phi_{f'|\gamma}\{\infty,0\}$, so $\Phi_{f'}(\gamma\{\infty,0\})=(\Phi_{f'|\gamma}\{\infty,0\})|\gamma^{-1}$. Theorem \ref{T:is_analytic} tells us $\Phi_{f'|\gamma}\{\infty,0\}\in\widetilde{\cD}(\Z_p^2)/\delta_0$. Since $\SL_2(\Z)$ acts on $\widetilde{\cD}(\Z_p^2)$, we conclude that $\Phi_{f'}(\gamma\{\infty,0\})\in\widetilde{\cD}(\Z_p^2)/\delta_0$.
\end{proof}

\subsection{The Vanishing Hypothesis}
In \cite{Ste1}, we described  conditions which guaranteed specializations of the Shintani cocycle produce measures. This \emph{vanishing hypothesis}, stated in terms of lines in $V$ can be reformulated in terms of the cusps $\P^1(\Q)=\P(V)$.

\begin{defn}
We say $f'\in\cS^{(p)}$ satisfies the vanishing hypothesis for a cusp $\frac{a}{b}\in\P^1(\Q)$, if for all vectors $w\in\Q^2$, the test function $x\mapsto f'\left(w +x{a \choose b}\right)\in\cS(\Q^{(p)})$ has Haar measure $0$.
\end{defn}
For a fixed $f'$, we will say a cusp $r\in\P^1(\Q)$ is \emph{good} for $f'$ if $f'$ satisfies the vanishing hypothesis for $s$.

It follows immediately from Lemma \ref{L:pole_desc} that, for all $\frac{a}{c},\frac{b}{d}\in\P^1(\Q)$, there exists $\mu\left\{\frac{a}{c},\frac{b}{d}\right\}\in\cD(\Z_p^2)$ such
\begin{equation*}
 	\Phi_{f'}\left\{\frac{a}{c},\frac{b}{d}\right\}= \frac{1}{(aX+cY)(bX+dY)}\cdot\mu\left\{\frac{a}{c},\frac{b}{d}\right\}
\end{equation*}

\begin{thm}\label{T:pole_description}
If $f'$ is good for $\frac{a}{c}$ or $\frac{b}{d}$, then there exists $\nu\in\cD(\Z_p^2)$ such that $\mu\left\{\frac{a}{c},\frac{b}{d}\right\}=(aX+cY)\cdot \nu$ or $(bX+dY)\cdot \nu$, respectively. In particular, $\Phi_{f'}\left\{\frac{a}{c},\frac{b}{d}\right\}\in\cD(\Z_p^2)$ if $f'$ is good for both cusps.
\end{thm}
\begin{proof}
This follows from the proofs of Proposition 4.14 and Theorem 4.15 of \cite{Ste1}.
\end{proof}

\section{The weight-$k$ specialization map}
In this section, we specialize the Shintani modular symbol to $\Symb_{\Gamma}(\cD^{\dagger}_k(\Z_p))$ by constructing $\Gamma_p$-equivariant maps from $\widetilde{\cD}(\Z_p^2)$ to Stevens' distribution modules $\cD_k(\Z_p)$ (see \cite{PolSte11}). It turns out that this construction will be simple for $k<0$, but non-trivial for non-negative $k$. Here $\cD_k(\Z_p)$ is the distribution module $\cD(\Z_p)$ a weight-$k$ action of the semigroup
\begin{equation*}
	\Sigma_0(p):=\left\{ \begin{pmatrix} a & b\\ c & d\end{pmatrix}\in M_2(\Z_p) : a\in\Z_p^\times, c\in p\Z_p, \text{ and } ad-bc\neq 0\right\}.
\end{equation*}
Fixing $k\in\Z$, $\Sigma_0(p)$ acts on the left on $\cA(\Z_p)$ by
\begin{equation*}
	\gamma\cdot_k f(z) = (a+cz)^k f\left(\frac{b+dz}{a+cz}\right)
\end{equation*}
for all  $\gamma=\begin{pmatrix} a & b\\ c & d\end{pmatrix}\in\Sigma_0(p)$. This induces a right action of $\Sigma_0(p)$ on $\cD(\Z_p)$ by $(\mu|_k \gamma)(f)=\mu(\gamma\cdot_k f)$. Our convention will be to write $\cA_k$, $\cA^\dagger_k$, $\cD^\dagger_k$, etc. to denote the corresponding modules with the above weight-$k$ actions. 

\subsection{Differential forms on wide opens}
Denote by $\cH_p$ the $p$-adic upper half-plane $\P^1(\C_p)-\P^1(\Q_p)$, endowed with the \emph{left} $\GL_2(\Q_p)$ action
\begin{equation*}
\gamma\cdot z= [(x,y)\gamma^*]= \frac{-b+az}{d-cz}.
\end{equation*}

Let $\cZ\subset\cH_p$ be the wide open $\cZ:=\{z\in\C_p : 1<|z|<p\}.$ Let $A(\cZ)$ the $\Q_p$-algebra of rigid analytic functions on $\cZ$. The space of K\"ahler differentials $\Omega(\cZ)$ on $\cZ$ is an $A(\cZ)$-module, and each of these spaces can be explicitly described in terms of the Laurent series that converge on $\cZ$. 
\begin{prop}\label{P:diffs}
Each analytic function $f\in A(\cZ)$ can be uniquely represented by a Laurent series 
\begin{equation*}
	f=\sum_{n\geq 0} a_n z^{-n} +\sum_{m> 0} b_m p^m z^m
\end{equation*}
Similarly, the differential forms $\omega\in\Omega_{}(\cZ)$ can be written uniquely as
\begin{equation} \label{E:diff_series}
	\omega = \sum_{n\geq 0} a_n z^{-n}\frac{dz}{z} +\sum_{m> 0} b_m p^m z^m \frac{dz}{z}
\end{equation}
Such an expression represents an analytic function (or differential form) if and only if the coefficients satisfy the following growth condition: For each $t>1$ $|a_k|=o(t^k)$ and $|b_{-k}|=o^(t^k)$ as $k\rightarrow\infty$
\end{prop}

Define $A^{bd}(\cZ)$ and $\Omega^{bd}(\cZ)$ to be the subsets of $A(\cZ)$ and $\Omega(\cZ)$ of elements whose Laurent series have ``bounded" coefficients, in the sense that the sequences $a_n$ and $b_n$ are bounded. One can show that
\begin{prop}
The $\Q_p$-vector space $A^{bd}(\cZ)$ is a Banach-algebra under the norm $|\sum_{n\geq0} a_nz^{-n}+\sum_{\geq 0} b_m p^mz^m|=\sup\{|a_n|_p,|b_n|_p\}$. Moreover, $\Omega^{bd}(\cZ)$ is $A(\cZ)$-module and $\Q_p$-Banach space.
\end{prop}

\begin{lem}
The subgroup $\Gamma_p\subset\Sigma_0(p)$ preserves $\cZ$, and thus acts on $\cA(\cZ),\cA^{bd}(\cZ)$ and $\Omega(\cZ),\Omega^{bd}(\cZ)$.
\end{lem}
\begin{proof}
Let $\gamma=\begin{pmatrix} a & b\\ c& d\end{pmatrix}\in\Gamma_p$, i.e. $c\in p\Z_p$ and $a,d\in\Z_p^\times$. Fixing $z\in\cZ$, we have $|\gamma\cdot z|_p=\frac{|b+az|_p}{|d-cz|_p}.$ Because $|az|_p=|z|_p>|b|_p$, we have $|b+az|_p=|z|_p$. Similarly, since $|cz|_p<1$ and $|d|_p=1$, $|d-cz|_p=|d|_p=1$. Therefore, $|\gamma z|_p=|z|_p$, so $\gamma\cdot z\in\cZ$. 
\end{proof}

\begin{defn}
For each integer $k$, the weight-$k$ action of $\Gamma_p$ on $\Omega_{}(\cZ)$ is defined by 
\begin{equation*}
	(\omega|_k\gamma)(z)=(a-bz^{-1})^k(\omega|_0\gamma).
\end{equation*}
Since $(a-bz^{-1})^k\in A^{bd}(\cZ)$, this extends to a well-defined action on $\Omega(\cZ)$ and $\Omega^{bd}(\cZ)$. We indicate that we are using this weight $k$ action with the subscript $k$, e.g. $\Omega_k(\cZ)$.
\end{defn}
\subsection{Modular symbols of differential forms}
%In this section, we pass from $\widetilde{\cD}(\Z_p^2)$ to analytic families of differential forms $\Omega_{}(\cZ)$. First, we ``$p$-stabilize" the Fourier transform of a distribution. 
Define, for each $\mu\in \widetilde{\cD}(\Z_p^2)$, the modified Fourier-transform
\begin{equation*}
	\cF^*(\mu):=\int_{\Z_p\times\Z_p^\times}\exp(xX+yY) d\mu(x,y)\in \widetilde{\cR}.
\end{equation*}
It is evident that $\cF^*:\widetilde{\cD}(\Z_p^2)\longrightarrow \widetilde{\cR}$ is a $\Q[X,Y]$-homomorphism, and it is $\Gamma_p$-equivariant. Moreover, $\cF^*(\delta_0)=0$, so $\cF^*$ induces a $\Gamma_p$-equivariant, $\Q[X,Y]$-homomorphism $\widetilde{\cD}(\Z_p^2)/\delta_0\longrightarrow \widetilde{\cR}$.

The ring $\Q[[X,Y]]$ and its localization $\widetilde{\cR}$ are graded by degree and  for each integer $k$ and element $F\in \widetilde{\cR}]$, denote by $F_k$ the degree $k$ component of $F$. It is a homogeneous rational polynomial of degree $k$. The space of degree $0$ rational polynomials in $(S^{-1}\Q[X,Y])_0$ is isomorphic to the ring of rational functions on $\P^1(\Q)$ by the (identical) maps $F(X,Y)\mapsto F(1,z)$ or $F(X,Y)\mapsto F(z^{-1},1)$. 

In the same way that a degree $0$ rational function $f(X,Y)$ can be identified with  differential form $f(1,z)dz$ or $f(z^{-1},1)dz^{-1}$, we can pass from degree $k$ rational functions to differential forms by $\pi_{k}:\widetilde{\cR}\longrightarrow\Omega_{}(\cZ)$,
\begin{equation*}
	\pi_k (F(X,Y)) = k! F_{k-2}(z^{-1},1)d{z^{-1}},
\end{equation*}\

\begin{lem}\label{L:S0_equivariant}
The projection $\pi_k$ induces a $\Gamma_p$-equivariant homomorphism $\pi_k :\widetilde{\cR} (1) \longrightarrow \Omega_{k} (\cZ)$.
\end{lem}
\begin{proof}
First, note that $dz^{-1}|_0\gamma = \det(\gamma) (a-bz^{-1})^{-2} dz^{-1}$. Therefore,
\begin{align*}
	\pi_k(\det(\gamma)F|\gamma)=&\det(\gamma)k!F_{k-2}( dX-cY,-bX+aY)dz^{-1}\\
	=&\det(\gamma)\left.(-bX+aY)^{k-2}k!F_{k-2}\left(\frac{dX-cY}{-bX+aY},1\right)\right|_{(X,Y)=(z^{-1},1)}dz^{-1}\\
	=&\det(\gamma)(a-bz^{-1})^{k-2}k!F_{k-2}\left(\frac{dz^{-1}-c}{-bz^{-1}+a}\right)dz^{-1}\\
	=&\pi_k(F)|_k\gamma.
\end{align*}
\end{proof}

Writing $\Gamma:=\Gamma_p\cap\Gamma_{f'},$ we have
\begin{cor}
The composition $\pi_k \circ\cF^*$ induces a homomorphism of modular symbols
\begin{equation*}
	\Symb_{\Gamma}(\widetilde{\cD}(\Z_p^2))(1)\xrightarrow{~\sim~}\Symb_{\Gamma}(\widetilde{\cD}(\Z_p^2)(1))\longrightarrow \Symb_{\Gamma}(\Omega_{k}(\cZ))
\end{equation*}
\end{cor}
\begin{proof}
The first isomorphism follows from that fact that $\Gamma_{f'}\subset\SL_2(\Z)$--thus a $\Gamma_{f'}$-invariant homomorphism remains $\Gamma_{f'}$-invariant after twisting the action by the determinant.

The second homomorphism is induced by the $\Gamma_p$-equivariant homomorphism of Lemma \ref{L:S0_equivariant}, since $\Gamma\subset \Gamma_p$.
\end{proof}

\begin{defn}
For each integer $k\geq0$,  define $\Phi_{f'}^k\in\Symb_{\Gamma}(\Omega_{k}(\cZ))$ to be the modular symbol $\pi_k\circ \Phi_{f'}$.
% be the image under $k!\pi_k\circ\cF^*$ of $\Phi$. Concretely, $$\Phi_{f'}^k(D) = k! (\Psi(D)(f'\otimes[\Z_p\times\Z_p^\times])_{k-2}|_{(X,Y)=(z^{-1},1)} dz^{-1}.$$
\end{defn}

\subsection{Analytic families of modular symbols}
In this section, we extend the weight $k$ specializations $\Phi_{f'}^k$ to analytic families of modular symbols over the $\Q_p$-points of weight space, $\cX_{\Q_p}=\Hom_{cts}(\Z_p^\times,\Q_p^\times)$. In fact we only build families over affinoid subdomains of $\cX_{\Q_p}$, but this restriction allows us to work in the comfortable realm of $\Q_p$-Banach spaces, avoiding entirely technical considerations of Fr\'echet spaces. We remark that it is possible to extend the family $\Phi_{f'}^\kappa$ to $\C_p$-points of $\cX$, but the resulting differential forms will converge on a smaller annulus.

The Teichm\"uller character $\omega_p$ expresses $\Z_p^\times=\mu_{p-1}\times (1+p\Z_p)$, so $\cX_{\Q_p}=\Hom_{cts}(\mu_{p-1},\Q_p^\times)\times\Hom_{cts}(1+p\Z_p,\Q_p^\times)=\Z/(p-1)\times \Hom_{cts}(1+p\Z_p,\Q_p^\times)$. For each $i\in\Z/(p-1)$, we write $\cX_i$ for the corresponding component $\Hom_{cts}(1+p\Z_p,\Q_p^\times)$ of weight-space. A choice of topological generator $\gamma$ of $1+p\Z_p$ identifies $\Hom_{cts}(1+p\Z_p,\Q_p^\times)=\Hom_{cts}(1+p\Z_p,\Z_p^\times)$ with $\Z_p$. Indeed, a continuous homomorphism $\kappa:1+p\Z_p\longrightarrow\Z_p^\times$ is determined by  $\kappa(\gamma)$  (which is necessarily in $1+p\Z_p$) so $\kappa(\gamma)=\gamma^{\alpha}$ for some unique $\alpha\in\Z_p$. We denote by $\wt(\kappa)$ the unique $p$-adic integer for which $\kappa(\langle \gamma\rangle )=\langle\gamma\rangle^{\wt(\gamma)}$; in other words, $\wt(\kappa)=\log_p(\kappa(\gamma))/\log_p(\gamma)$. Thus, $\cX_{\Q_p}$ can be viewed as $p-1$ copies of $\Z_p$, and the (rigid) analytic functions on $A(\cX_{\Q_p})=\bigoplus_{i=1}^{p-1}A(\cX_i)\cong\bigoplus_{i=1}^{p-1}A[\Z_p,1]$. Finally, we note that the integers embed in $\cX_{\Q_p}$ by sending $k$ to the $k$-th power map. Under this embedding, $\wt(k)=k$.

\begin{lem}
For each $\kappa\in\cX_{\Q_p}$, the function $\kappa:\Z_p^\times\longrightarrow\Q_p^\times$ extends to a rigid analytic function $\kappa\in A[\Z_p,r]$ for any $r<1$.
\end{lem}
\begin{proof}
The restriction of $\kappa$ to a given residue disc $\omega_p(a)+p\Z_p$ can be described by the binomial theorem,
\begin{align}
	\kappa(\omega_p(a)+u) &= \kappa(\omega_p(a))\kappa(1+u/\omega_p(a))\\
	&=\kappa(\omega_p(a))(1+u/\omega_p(a))^{\wt(\kappa)}\\
	&=\kappa(\omega_p(a))\sum_{n\geq 0} {\wt(\kappa) \choose n} u^n\omega_p(a)^{-n},\label{E:binom}
\end{align}
where ${\wt\kappa \choose n}$ is the usual binomial coefficient extended to $\Z_p$. Now each ${\wt\kappa \choose n}$ has $p$-adic valuation bounded above by $1$ (crucial to this is $\kappa\in\cX_{\Q_p}$), so the series (\ref{E:binom}) converges for $|u|<1$. Thus, $\kappa\in A[\Z_p,r]$ for $r<1$.
\end{proof}

The Lemma allows us to extend the weight-$k$ automorphy factor $(a-bz^{-1})^k$ to arbitrary weights $\kappa\in\cX_{\Q_p}$ by putting
\begin{equation*}
	(a-bz^{-1})^\kappa:=\kappa(\omega_p(a))\left(\frac{a-bz^{-1}}{\omega_p(a)}\right)^{\wt(\kappa)}.
\end{equation*}
This defines, for each $\kappa\in\cX_{\Q_p}$, a weight-$\kappa$ action of $\Gamma_p$ on $\Omega(\cZ)$ by
\begin{equation*}
	\omega |_\kappa \gamma :=(a-bz^{-1})^\kappa \omega|_0\gamma
\end{equation*}
Moreover, one can define an action of $\Gamma_p$ on families (see \cite{Bel12}, \S 3.1)--for each affinoid $W_{\Q_p}\subset \cX_{\Q_p}$, there is an action of $\Gamma_p$ on $\Omega_{}^{bd}(\cZ)\widehat{\otimes}_{\Q_p}A(W_{\Q_p})$ specializing at each $\kappa\in W_{\Q_p}$ to the weight-$\kappa$ action on $\Omega_{}^{bd}(\cZ)$. 

The following theorem generalizes earlier results of Campell \cite{Duff} and Kostadinov \cite{Kalin}:

\begin{thm}
For each affinoid $W\subset \cX$, there exists an analytic family of modular symbols 
\begin{equation*}
\mathbf{\Phi}_{f'}^W\in\Symb_{\Gamma}(\Omega^{bd}(\cZ)\widehat\otimes_{\Q_p}A(W_{\Q_p}))
\end{equation*} 
whose evaluation at each non-negative integer $k\in W\cap \Z_{\geq0}$ is $\Phi_{f'}^k$.
\end{thm}
\begin{proof}

Fix an arbitrary divisor $D=\left\{\frac{a}{c},\frac{b}{d}\right\}\in\Delta_0$ with $ad-bc\neq0$ and $W_{\Q_p}\subset \cX_{\Q_p}$ an affinoid. Our first step is to exhibit $\Phi_{f'}^k(D)$ as the weight-$k$ specialization of an analytic family in $\Omega^{bd}_{}(\cZ)\widehat\otimes_{\Q_p}A(W_{\Q_p})$. Recall that Theorem \ref{T:is_analytic} states $\Phi_{f'}(D)=\frac{1}{(aX+cY)(bX+dY)}\mu_{f'}(D)$ for some $\mu_{f'}(D)\in\cD(\Z_p^2)$. For each integer $k\geq 0$, the degree $k-2$ part of $\cF^*(\Phi_{f'}(D))$ is the degree $k$ part of $\cF^*(\mu_{f'}(D))$ multiplied by $\frac{1}{(aX+cY)(bX+dY)}$. Therefore,
\begin{equation}\label{E:weight_k}
	\Phi_{f'}^k(D) = \left(\int_{\Z_p\times\Z_p^\times}(xz^{-1}+y)^{k} d\mu_{f'}(D)\right) \frac{dz^{-1}}{(az^{-1}+c)(bz^{-1}+d)}.
\end{equation}
The differential form $\frac{dz^{-1}}{(az^{-1}+c)(bz^{-1}+d)}$ belongs to $\Omega^{bd}_{}(\cZ)$ and the polynomial 
\begin{equation*}
\int_{\Z_p\times\Z_p^\times}(xz^{-1}+y)^kd\mu_{f'}(D)=\sum_{n=0}^k{k\choose n} \left(\int_{\Z_p\times\Z_p^\times}(x/y)^ny^k d\mu_{f'}(D)\right) z^{-n}
\end{equation*}
 is an element of $A^{bd}(\cZ)$, and thus their product, $\Phi_{f'}^k(D)$, is a differential form in $\Omega^{bd}_{}(\cZ)$. For each integer $n\geq 0$, the coefficient of $z^{-n}\frac{dz^{-1}}{(az^{-1}+c)(bz^{-1}+d)}$ is the weight $k$ specialization of the analytic  function $\alpha_n\in A(\cX)$ defined by $\alpha_n(\kappa):= {\wt{\kappa} \choose n}\int_{\Z_p\times\Z_p^\times} (x/y)^n y^\kappa d\mu_{f'}(D).$ We will show that, for each $\kappa\in W_{\Q_p}$, the coefficients $\alpha_n(\kappa)$ are bounded and thus $\sum_{n\geq 0} \alpha_n(\kappa)z^{-n}\frac{dz^{-1}}{(az^{-1}+c)(bz^{-1}+d)}$ is a differential form in $\Omega^{bd}_{}(\cZ)$. This will show $\sum_{n\geq 0} \alpha_nz^{-n}\frac{dz^{-1}}{(az^{-1}+c)(bz^{-1}+d)}\in\Omega^{bd}(\cZ)\widehat{\otimes}_{\Q_p}A(W_{\Q_p})$.

Fixing $\kappa\in W\subset \cX_{\Q_p}$, observe that the factor ${\wt{\kappa} \choose n}$ is $p$-adically bounded as $n$-varies.  The interesting part of $\alpha_n(\kappa)$ is $\int_{\Z_p\times\Z_p^\times} (x/y)^n y^\kappa d\mu_{f'}(D)$, but our work in showing that $\mu_{f'}(D)$ is a locally analytic distribution will finally pay dividends. Observe that $(x/y)^ny^{\kappa}|_{\Z_p\times\Z_p^\times}$ is a rigid analytic function of some radius $r$, i.e. an element of $A[\Z_p^2,r]$. One sees that the sup-norm $||(x/y)^ny^{\kappa}|_{\Z_p\times\Z_p^\times}||_{r}$ is bounded above by $1$, as the image of $\kappa$ belongs to $\Z_p^\times$. Therefore, it suffices to consider our locally analytic distribution $\mu_{f'}(D)$ restricted to the $\Q_p$-Banach space of rigid analytic distributions $D[\Z_p^2,r]$. Let $||\mu_{f'}(D)||_{r}$ denote the Gauss norm of $\mu_{f'}(D)$ restricted to $D[\Z_p^2,r]$, i.e. 
\begin{equation*}
	||\mu_{f'}(D)||_{r}=\sup_{\substack{g\in A[\Z_p^2,r]\\ g\neq 0}} \frac{ |\mu_{f'}(D)(g)|_p}{||g||_{r}}.
\end{equation*}
It follows that 
\begin{equation*}
\left|\int_{\Z_p\times\Z_p^\times} (x/y)^ny^{\kappa }d\mu_{f'}\right|_p \leq ||\mu_{f'}(D)||_{1/p}\cdot || (x/y)^ny^{\kappa }|_{\Z_p\times\Z_p^\times}||_{r}\leq ||\mu_{f'(D)}||_{r}. 
\end{equation*}

Since the values of $\Phi_{f'}^k$ interpolate into analytic families in $\Omega^{bd}_{}(\cZ)\widehat{\otimes}_{\Q_p} A(W_{\Q_p})$, we have a map $\mathbf{\Phi}^{W}_{f'}:\Delta_0\longrightarrow \Omega^{bd}_{}(\cZ)\widehat{\otimes}_{\Q_p} A(W_{\Q_p})$. It only remains to show that $\mathbf{\Phi}^W_{f'}$ satisfies the necessary conditions to define a modular symbol, but this follows from a simple continuity argument. For example, to show that $\mathbf{\Phi}^W_{f'}$ is $\Gamma$-invariant, fix an arbitrary $D\in \Delta_0$ and $\gamma\in \Gamma$. Consider $\mathbf{\Phi}^W_{f'}(\gamma D)|\gamma- \mathbf{\Phi}^W_{f'}\in \Omega^{bd}_{}(\cZ)\widehat{\otimes}_{\Q_p} A(W_{\Q_p})$, which we wish to show is zero. The description of $\Omega^{bd}(\cZ)$ in terms of Laurent series allows us to regard $\Omega_{}(\cZ)\widehat{\otimes}_{\Q_p}A(\cX)$ as a subset of $A(W_{\Q_p})[[z,z^{-1}]]dz^{-1}$. To show that $\mathbf{\Phi}^W_{f'}(\gamma D)|\gamma- \mathbf{\Phi}^W_{f'}=\sum_{n\in\Z}\alpha_n z^{-n}dz^{-1}$ is zero, it is enough to show that each coefficient $\alpha_n\in\cA(W_{\Q_p})$ vanishes. Evaluating at an integer $k\geq 0$ in $W$,
\begin{equation*}
	\ev_k(\mathbf{\Phi}^W_{f'}(\gamma D)|\gamma)-\ev_k(\mathbf{\Phi}^W_{f'})=\Phi_{f'}^k(\gamma D)|_k\gamma-\Phi_{f'}^k=0,
\end{equation*}
so each coefficient $\alpha_n$ vanishes on a dense subset of $\cX_{\Q_p}$. It follows that each $\alpha_n$ is identically $0$, and thus $\mathbf{\Phi}_{f'}$ is $\Gamma$-invariant. The same arguments show that $\mathbf{\Phi}^W_{f'}$, extended by linearity, satisfies the necessary relations, and thus defines a modular symbol.
\end{proof}

Once we have the analytic families $\mathbf{\Phi}_{f'}^W$, we can \emph{define} the specialization of $\Phi_{f'}$ to negative weights $-k\leq 0$ by picking a an affinoid neighborhood $W$ of $-k$ and evaluating $\mathbf{\Phi}_{f'}^W$ at weight $-k$: $\Phi_{f'}^{-k}:= \ev_{-k}(\mathbf{\Phi}^W_{f'}) \in\Symb_{\Gamma}(\Omega_{-k}^{bd}(\cZ))$.

\begin{lem}\label{L:is_zero}
If the test function $f'$ is good for a cusp $r$ (and hence for all the cusps in $\Gamma\cdot r$), then the analytic family of modular symbols $\Phi^{k}_{f'}$ will vanish at $k=0$ and $\lim_{k\mapsto 0}\frac{1}{k} \Phi_{f'}^k$ is a well-defined modular symbol. 
\end{lem}
\begin{proof}
For any cusp $s=\frac{b}{d}$ distinct form $r$, the pseudo-distribution is of the form $\Phi_{f'}\{r,s\}=\frac{1}{bX+dY}\cdot\mu_{f'}\{r,s\}$ by Theorem \ref{T:pole_description}. Therefore, the degree $k>0$ part of $\Phi_{f'}\{r,s\}$ is equal to 
\begin{equation}
	\Phi_{f'}\{r,s\}= k \left(\int_{\Z_p\times\Z_p^\times} (xz^{-1}+y)^{k-1}d \mu_{f'}\{r,s\}\right)\frac{dz^{-1}}{bz^{-1}+d}.
\end{equation}
Taking $k\rightarrow 0$, we see that $\Phi_{f'}^0$ vanishes on all divisors supported at the good cusp $r$ and that $\lim_{k\rightarrow 0} \frac{1}{k}\Phi_{f'}^k$ is well-defined. We conclude by noting that $\Delta_0$ is generated by such divisors.
\end{proof}

\subsection{The residue pairing}

%Let $S\subset \P^1(\Q_p)$ be a compact open and $r\in \R$ is a positive real number--write $B[S,r]$ as the finite disjoint union of balls $B[S,r]=\bigcup_{i=1}^n B[s_i,r]$. Suppose $G\subset\GL_2(\Q)$ stabilizes each of the balls $B[s_i,r]$, under the left action ($\gamma\cdot z\mapsto \frac{-b+az}{d-cz}$). Letting $G$ act on the $\Omega_{}(\cW)$ via the standard weight $0$  action, 
\begin{lem}\label{L:def_res}
There is a canonical $\Gamma_p$-invariant homomorphism, the residue map, $\Res:\Omega_{}(\cZ)\longrightarrow \Q_p$, characterized by
\begin{align*}
	\Res: \sum_{n\in \Z} a_n z^{-n}\frac{dz}{z}=a_0
\end{align*}
\end{lem}
\begin{proof}
See the Remark on pg. 223 of \cite{Sch84}.
\end{proof}

%Thus the residue map gives us a bilinear pairing $\langle~,~\rangle_0:\cA^\dagger(\Z_p,1)\times\Omega(\cW_0)\longrightarrow\Q_p$. 
%Equipping $\cA^\dagger$ and $\Omega(\cZ)$ with the weight-$0$ actions of $\Sigma_0(p)$, it is easy to check (and we will show below) that $\langle \gamma^*\cdot f,\omega|\gamma\rangle=\langle f,\omega \rangle$. This can be extended to non-zero weights by a slight modification of the pairing: 
For each $k\in\Z$, we define a pairing $\langle~,~\rangle_k:\Omega_{-k}(\cZ)\times \cA^\dagger_k(\Z_p,1)\longrightarrow\Q_p$ by $\langle \omega,f\rangle_k:=\Res (z^{-k}f \omega).$
\begin{prop}
The bilinear pairing $\langle~,~\rangle_k:\Omega_{-k}(\cZ)\times \cA^\dagger_k(\Z_p,1)\longrightarrow\Q_p$ is $\Gamma_p$-equivariant: $\langle\omega|_{-k}\gamma, \gamma^*\cdot_k f\rangle =\langle \omega,f\rangle$.
\end{prop}
\begin{proof}
Fix $\omega\in\Omega_{-k}(\cZ)$ and $f\in\cA^\dagger(\Z_p,1)$.  An element $\gamma\in\Gamma_p$ acts by
\begin{equation*}
	\omega|_{-k}\gamma(z) = (a-bz^{-1})^{-k}\omega(z|\gamma^*) \text{ and } \gamma\cdot_k f(z) = (a+cz)^kf(z|\gamma)
\end{equation*}
%Applying $\gamma^*$ to $f$, we get $\gamma^*\cdot_k f(z) =(d-cz)^k f(z|\gamma^*)$. 
Because $\gamma$ stabilizes $B[\Z_p,1]$, we have
\begin{align*}
	\langle \omega|_{-k}\gamma, \gamma^*\cdot_kf(z)\rangle_k&\\
	 = &\Res_{B[\Z_p,1]}(z^{-k} (a-bz^{-1})^{-k}(d-cz)^k f(z|\gamma^*)\omega(z|\gamma^*))\\
	=&\Res_{B[\Z_p,1]}\left( \frac{ az-b}{d-cz}\right)^{-k} f(z|\gamma^*)\omega(z|\gamma^*)\\
	=&\Res_{B[\Z_p,1]} (z|\gamma^*)^{-k}f(z|\gamma^*)\omega(z|\gamma^*)\\
	=&\Res_B[\Z_p,1]{} z^{-k}f(z)\omega(z)\\
	=&\langle \omega, f\rangle_k.
\end{align*}

%\begin{equation*}
%	(\omega|_{-k}\gamma) (\gamma^*\cdot_kf) = (\omega f)(z|\gamma^*).
%\end{equation*}
%Since $\Sigma_0(p)$ preserves the disks $B[\Z_p,1]$ and $B[\bx_\infty,1/p]$, Lemma \ref{L:def_res} shows $\Res_{}( \omega f)(z|\gamma^*)=\Res_{}(\omega f)$. Therefore,
%\begin{equation*}
%	\mu_{\omega|_{-k} \gamma} (f) = \mu_{\omega} (\gamma^{*-1}\cdot_k f) 
%\end{equation*}
\end{proof}
In this way, each $\omega\in\Omega_{-k}(\cZ)$ gives us a continuous linear functional $\mu_\omega:\cA^\dagger(\Z_p,1)\longrightarrow\Q_p$ by $\mu_{\omega}(f)= \Res(f \omega).$  It follows from the above lemma that $$\mu_{\omega|_{-k}\gamma}(f)=\langle \omega|_{-k} \gamma,f\rangle_k=\langle \omega,\gamma^{*-1}\cdot_k f\rangle_k=\det^{-k}(\gamma)\langle \omega,\gamma\cdot_k f\rangle_k=\det^{-k}\gamma \mu_{\omega}(\gamma\cdot_k f).$$ In other words, $\det^{k}(\gamma)\mu_{\omega|_{-k}\gamma}= \mu_{\omega}|_k\gamma$. Therefore, we get a $\Gamma_p$-equivariant map
\begin{equation*}
	\Omega_{-k}(\cZ)(k)\longrightarrow \cD^\dagger_k(\Z_p),
\end{equation*}
where the $(k)$ denotes we have twisted the action of $\Gamma_p$ by $\det^k$. In fact, we can use the boundedness of the coefficients of $\omega\in\Omega^{bd}(\cZ)$ to show, via Proposition 3.1 of \cite{PolSte}, that
\begin{cor}
For each $k\in\Z$, the residue pairing induce a $\Gamma_p$-equivariant homomorphism $\Omega^{bd}_{-k}(\cZ)(k)\longrightarrow\mathbf{D}_k$ by $\omega\mapsto\mu_\omega$.   \end{cor}

\begin{defn}
Denote by $R:\Symb_{\Gamma}(\Omega_{-k}^{bd}(\cZ)(k))\longrightarrow\Symb_{\Gamma}(\bD_k)$ the $\Gamma_p$-homomorphism induced by $\omega\mapsto\mu_\omega$.
\end{defn}

We emphasize that these are \emph{not} $\Sigma_0(p)$-equivariant homomorphisms. The $\Gamma_p$-equivariance is enough to compare the tame Hecke action on both spaces of modular symbols, but the actions of $U_p$ are not directly comparable. The next lemma allows us to compute the values of $R\Phi_{f'}|U_p$ in terms of $\Phi_{f'}|U_p$, at least on divisors where the only poles of $\Phi_{f'}$ are well-controlled.

\begin{lem}\label{L:almost_equivariance}
Let $\gamma\in\Sigma_0(p)$ and suppose $\omega\in\Omega(\cZ)$ has no poles in the annulus $B[\Z_p,1]|\gamma^{-1}-B[\Z_p,1]$. Then, for all $f\in\cA^\dagger(\Z_p)$, $\langle \omega, \gamma\cdot_k f\rangle_k=\det(\gamma)^k\langle \omega|_{-k}\gamma,f\rangle_k$. Put another way, $\mu_\omega|_k\gamma=\det(\gamma)^k\mu_{\omega|_{-k}\gamma}$.
\end{lem}
\begin{proof}
The remark on pg. 223 of \cite{Sch84} shows $\Res_{B}(\omega(z|\gamma))=\Res_{B|\gamma^{-1}}(\omega)$. If $\omega$ is holomorphic on $B|\gamma^{-1}-B$, then $\Res_{B|\gamma^{-1}}(\omega)=\Res_{B}(\omega)$.
\end{proof}

\section{Eisenstein symbols}

\subsection{Hecke action}

All computations of Hecke operators will reduce to computing the Hecke action on test functions. For example, a tame Hecke operator $T$ can be written as a sum of double coset representatives $\beta_1,\ldots,\beta_n\in \Gamma_p$, and we get
\begin{equation*}
	\Phi_{f'}|T=\sum_{i=1}^n \Phi_{f'}|\beta_i=\sum_{i=1}^n \Phi_{f'|\beta_i^{*-1}}
\end{equation*}

Thus, we turn to the groups of test functions $\cS(V)$, $\cS(V^{(p)})$ and $\cS(V_p)$ equipped with the adjugate-inverse action. Let $\Gamma$ be a congruence subgroup of $\GL_2(\Q)$, and let $\cS(V)^\Gamma$ denote the subgroup of test functions invariant with respect to $\Gamma$. Then $\cS(V)^\Gamma$ has an action of the Hecke module via the adjugate-inverse action: If $T$ is represented by the coset $\bigcup_{i=1}^r \Gamma \beta_i$, then we write $f|T^{*-1}:=\sum_{i=1}^r f|\beta_i^{*-1}.$ For each $\lambda\in \Q^\times$, we will write $[ \lambda ] :=\begin{pmatrix} \lambda & 0 \\ 0 & \lambda \end{pmatrix}$. Observe that $[ \lambda ] ^* =[ \lambda ]$ and $[ \lambda ]^{-1} = [ \lambda^{-1} ]$. Finally, let $\iota$ denote the Hecke operator represented by $\begin{pmatrix} 1 & 0 \\ 0 & -1\end{pmatrix}$.

First, we consider Hecke eigenvectors for $\Gamma_0(N)$. Observe that, for each prime $q$, the test function $[\Z_q^2]\in\cS(V_q)$ is fixed by $\Gamma_0(N)$. Letting $T_q$, $U_q$ denote the usual Hecke operators for $\Gamma_0(N)$, we have:
\begin{lem}\label{L:Hecke}
For any prime $q$, $[\Z_q^2]|T_q^{*-1}= q[\Z_q^2]|[ q]^{*-1}+[\Z_q^2].$
\end{lem}
\begin{proof}
\begin{align*}
	[\Z_q^2]|T_q^{*-1} 
	=& \sum_{a=0}^{q-1} [\Z_q^2] \left| \begin{pmatrix} 1 & a \\ 0 & q \end{pmatrix}^{*-1}\right.+ [\Z_q^2]\left| \begin{pmatrix} q & 0\\ 0 & 1 \end{pmatrix}^{*-1}\right.\\
	=&\sum_{a=0}^{q-1} \left[ \begin{pmatrix} q & -a\\ 0 & 1\end{pmatrix}\Z_q^2\right] +\left[ \begin{pmatrix} 1 & 0 \\ 0 & q\end{pmatrix}\Z_q^2\right]\\
	=&\sum_{a=0}^{q-1}\sum_{i=0}^{q-1} \left[ \begin{pmatrix} q & -a\\ 0 & 1\end{pmatrix}\left(\begin{pmatrix} 0\\i \end{pmatrix}+\Z_q\times q\Z_q\right)\right] + [\Z_q\times q\Z_q]\\
	=&\sum_{a=0}^{q-1}\sum_{i=0}^{q-1} \left[ \begin{pmatrix} -ai\\i \end{pmatrix}+q\Z_q\times q\Z_q\right] + [\Z_q\times q\Z_q]\\
	=& q[q\Z_q^2]+[\Z_q]\\
	=& q[\Z_q^2]|[ q]^{*-1}+[\Z_q^2]
\end{align*}
\end{proof}
A similar calculation shows  $[\Z_p\times p\Z_p]|U_p^{*-1} = p[(p\Z_p)^2]$ and $[\Z_p\times \Z_p]|U_p^{*-1}= p[(p\Z_p)^2]+[\Z_p\times\Z_p^\times]$. Together, these imply
\begin{lem}\label{L:U_p}
$[\Z_p\times\Z_p^\times]|U_p^{*-1}=[\Z_p\times\Z_p^\times]$.
\end{lem}

Given a prime $\ell$ distinct from $p$, we will write $f_\ell'$ for the test function
\begin{equation*}
	f'_\ell=\left(\bigotimes_{q\neq p,\ell} [\Z_q^2] \right)\otimes([\Z_\ell^2]-\ell[\Z_\ell\times\ell\Z_\ell]),
\end{equation*}
which factorizes as $f_1\times f_2$, with
\begin{equation*}
	f_1'=\bigotimes_{q\neq p} [\Z_q] \text{ and } f_2' =\left(\bigotimes_{q\neq p,\ell}[\Z_q]\right)\otimes([\Z_\ell]-\ell[\ell \Z_\ell]).
\end{equation*}
%$[\Z^{(p)}]\otimes([\Z^{(p)}]-\ell[\ell\Z^{(p)}]$, using the shorthand $[\Z^{(p)}]:=\bigotimes_{q\neq p} [\Z_q]$.
Observe that the stabilizer of $f'_\ell$ is $\Gamma_0(\ell)$; the previous lemmas show
\begin{prop}
The test function $f'_{\ell}$ is a Hecke-eigenvector with eigenvalues
\begin{enumerate}
	\item $f'_\ell |T_q^{*-1}=f_\ell'+qf_\ell'|[q]^{*-1}$ for $q\nmid \ell p$
	\item $f'_\ell|\iota^{*-1}=f_\ell'$.
\end{enumerate}
\end{prop}

%Now suppose $\Gamma=\Gamma_1(N)$, $p\nmid N$. Let $m|M$ and define, for $i\in\Z/M\Z$ and $j\in \Z/(M/m)\Z$,
%\begin{equation*}
%	f'_{i,j,m} := \left[ \begin{pmatrix} i \\ jm\end{pmatrix} +m\Z \times M\Z \right].
%\end{equation*}
%
%\begin{prop}
%The test functions $f'_{i,j,m}$ are invariant under $\Gamma_1(M)$. Moreover,
%\begin{enumerate}
%	\item $f'_{i,j,m}|\langle a \rangle = f'_{a^{-1}i,aj,m}$ for all $a\in(\Z/M\Z)^\times$.
%	\item $f'_{i,jm}|T_q=$? for all $q\nmid Mp$.
%	\item $f'_{i,jm}|U_q=$?? for $q\nmid M$.
%\end{enumerate}
%\end{prop}

Next, we turn to the congruence groups $\Gamma_1(N)$. Let $\tau$ be a primitive Dirichlet character of conductor $L$, $\psi$ a primitive dirichlet character of conductor $M$, and let $N=ML$. Define a test function $f'_{\tau,\psi}$ away from $p$ by specifying
\begin{equation}\label{E:filled_in}
	f'_{\tau,\psi}\otimes[\Z_p^2]=\sum_{\substack{i=1\ldots L\\ j=1\ldots M}} \tau^{-1}(j)\psi(i)\left[\begin{pmatrix} i\\ jL\end{pmatrix} + L\Z\times N\Z\right].
\end{equation}
\begin{rem}\label{R:factorization}
The test function $f'_{\tau,\psi}\otimes[\Z_p^2]$ is factorizable: $f'_{\tau,\psi}\otimes[\Z_p^2]=\left(\sum_{i=1}^M\tau^{-1}(i)[i+L\Z]\right)\times \left(\sum_{j=1}^M \psi(j)[jL + N\Z]\right)$. It follows that $f_{\tau,\psi}'$ is factorizable, a fact  we will use when computing $p$-adic $L$-functions.
\end{rem}
It is not hard to see that $f'_{\tau,\psi}$ is fixed by $\Gamma_1(N)$. Indeed, each summand of the honest test function in (\ref{E:filled_in}) is $\Gamma_1(N)$-invariant. A lengthy but elementary calculation shows that $f'_{\tau,\psi}$ is a Hecke-eigenvector with eigenvalues corresponding to the Eisenstein series $E_{\tau,\psi}$. We record the statement but omit the proof. %See, \cite{Ste82}
\begin{prop}
The test function $f'_{\tau,\psi}$ is a Hecke-eigenvector with eigenvalues
\begin{enumerate}
	\item $f'_{\tau,\psi} |\langle a \rangle^{*-1} = \tau(a)\psi(a)f'_{\tau,\psi}$
	\item $f'_{\tau,\psi}|T_q^{*-1}=\tau(q)f_{\tau,\psi}+q\psi(q)f_{\tau,\psi}|[q]^{*-1}$, for $q\nmid Np$.
	%\item If $q| N$, $f'_{\tau,\psi}|U_q^{*-1}=\tau(q)f'_{\tau,\psi}+q\psi(q)f_{\tau,\psi}|[q]^{*-1}$
	\item $f'_{\tau,\psi}|\iota^{*-1}=\psi(-1)pf'_{\tau,\psi}$
\end{enumerate}
\end{prop}

Let us write $\Phi_{\ell}:=\Phi_{f'_\ell}$ and $\Phi_{\tau,\psi}:=\Phi_{f'_{\tau,\psi}}$. Tracing these eigenvalues through the $\Gamma_p$-equivariant compositions $\Symb_\Gamma(\widetilde{\cD}(\Z_p^2)/\delta_0)(k+1)\rightarrow\Symb_\Gamma(\Omega_{-k}^{bd}(\cZ))(k)\rightarrow\Symb_{\Gamma}(\mathbf{D}_k)$, we arrive at the corollary
\begin{cor}
Let $k\geq 0$ be an integer. The modular symbol $R\Phi_{\ell}^{-k}\in\Symb_{\Gamma_0(p\ell)}(\mathbf{D}_k)$ has Hecke eigenvalues
\begin{enumerate}
	%\item $R\Phi_{{\tau,\psi}}^{-k}a |\langle a \rangle = \tau(a)\psi(a)R\Phi_{{\tau,\psi}}^{-k}$
	\item[(1a)] $R(\Phi_{\ell}^{-k})|T_q=R\Phi_{\ell}^{-k}+q^{k+1}R\Phi_{\ell}^{-k}$, for $q\nmid  \ell p$.
	\item[(2a)] $R\Phi_{\ell}^{-k}|\iota=(-1)^{k+2} R\Phi_{\ell}^{-k}$. 
	%\item If $q| N$, $f'_{\tau,\psi}|U_q^{*-1}=\tau(q)f'_{\tau,\psi}+q\psi(q)f_{\tau,\psi}|[q]^{*-1}$
	%\item $f'_{\tau,\psi}|U_p^{*-1}=\tau(p)pf'_{\tau,\psi}$
\end{enumerate}
The modular symbol $R\Phi_{{\tau,\psi}}^{-k}\in\Symb_{\Gamma_1(N)\cap\Gamma_0(p)}(\mathbf{D}_k)$ has Hecke eigenvalues
\begin{enumerate}
	\item[(1b)] $R\Phi_{{\tau,\psi}}^{-k}a |\langle a \rangle = \tau(a)\psi(a)R\Phi_{{\tau,\psi}}^{-k}$
	\item[(2b)] $R(\Phi_{{\tau,\psi}}^{-k})|T_q=\psi(q)R\Phi_{\tau,\psi}^{-k}+q^{k+1}\tau(q)R\Phi_{{\tau,\psi}}^{-k}$, for $q\nmid  Np$.
	\item[(3b)] $R\Phi_{\tau,\psi}^{-k}|\iota=(-1)^{k+2}\psi(-1) R\Phi_{\tau,\psi}^{-k}$. 
	%\item If $q| N$, $f'_{\tau,\psi}|U_q^{*-1}=\tau(q)f'_{\tau,\psi}+q\psi(q)f_{\tau,\psi}|[q]^{*-1}$
	%\item $f'_{\tau,\psi}|U_p^{*-1}=\tau(p)pf'_{\tau,\psi}$
\end{enumerate}
\end{cor}
\begin{proof}
We only show the second half of the corollary, the first following from the same arguments: Claim (1b) is immediate. In order to see (2b), note that each of the double coset representatives for $T_q$ have determinant $q$. Therefore, $R\Phi_{\tau,\psi}^{-k}|T_q=R(q^k \Phi_{\tau,\psi}^{-k}|T_q)=R\pi_{-k}(q^{k+1}\Phi_{\tau,\psi}|T_q)$. Using the fact that  $\Phi_{f'|[q]^{*-1}}^{-k}=q^{-k-2}\Phi_{f'}^{-k}$, we get $R(\Phi_{\tau,\psi}^{-k}|T_q)= q^{k+1}( \tau(q)R\Phi_{\tau,\psi}^{-k} +q^{k-2}q\psi(q) R\Phi_{\tau,\psi}^{-k})=\psi(q)R\Phi_{\tau,\psi}^{-k}+q^{k+1}\tau(q)R\Phi_{{\tau,\psi}}^{-k}$, as desired. The third claim follows in the same way, remembering that the action of $\GL_2(\Q)$ on $\widetilde{\cR}$ is twisted by $\sign(\det)$: So $\Phi_{\tau,\psi}^{-k}|\iota = (-1)^{k+1}\sign(-1)\Phi_{f'_{\tau,\psi}|\iota}^{-k} = (-1)^{k+2}\psi(-1)\Phi_{\tau,\psi}$.
\end{proof}

Now we turn to $U_p$:  For each $i=1,\ldots,p$, write $\beta_i$ for the usual double coset representative $\beta_i=\begin{pmatrix} p & i \\ 0 & 1\end{pmatrix}$. It should be pointed out that applying $\beta_i$ to a general differential form $\omega\in\Omega_{k}(\cZ)$ gives us a differential form on the annulus $\frac{1}{p}\beta_i\cZ$, and these annuli will have no common points. Thus it is not clear how to define an action of $U_p$ on $\Symb_{\Gamma}(\Omega_{k}(\cZ))$. However, a consequence of the next lemma is that applying $\beta_i$ to the values of $\Phi_{f'}^k$ gives differential forms which do converge on the annulus $\cZ$, and we can compute $\Phi_{f'}^k|U_p$.

\begin{lem}
Let $k\in\Z$ be any integer. If there is some $\lambda$ for which $f'|\beta_i^{*-1}=\lambda f'$ for all $i=1,\ldots,p$, then,  $\Phi_{f'}^k$ is a $U_p$-eigensymbol with eigenvalue $p\lambda$. 
\end{lem}
\begin{proof}
Before peeling back the layers of notation, let us emphasize that we should view the computations at $p$ as conceptually identical to the computations away from $p$: the fact that our $p$-adic pseudo-distribution comes from a global distribution allows us to move all the Hecke computations back to test functions. 

To fix ideas, pick a positive integer $k>0$ and a divisor $D$. Applying $U_p$ to $\Phi_{f'}^k\in \Symb_{\Gamma}(\Omega^{bd}(\cZ))(1)$ and evaluating at $D$ gives
\begin{align*}
	(\Phi_{f'}^k|U_p)(D)=\sum_{i=1}^p \det(\beta_i)\Phi_{f'}^k(\beta_i D) |_k\beta_i \\
					=\sum_{i=1}^p\det(\beta_i)\pi_k(\cF^*(\Phi_{f'}(\beta_i D))|\beta_i)\\
					=p\cdot\pi_k\left(\sum_{i=1}^p\cF^*(\Phi_{f'}(\beta_i D))|\beta_i\right).
\end{align*}
Unfolding definitions, $\cF^*\Phi_{f'}(\beta D)|\beta =\Psi(\beta D)(f'\otimes[\Z_p\times\Z_p^\times]) |\beta$; the fact that $\Psi$ is a $\GL_2^+(\Q)$-modular symbol gives $\Psi(\beta D)(f'\otimes [\Z_p\times\Z_p^\times]) |\beta=\Psi(D)(f'\otimes[\Z_p\times\Z_p^\times]|\beta^{*-1})$. Our hypothesis on $f'$ tells us $(f'\otimes[\Z_p\times\Z_p^\times])|\beta_i^{*-1}=\lambda f'\otimes([\Z_p\times\Z_p^\times]|\beta_i^{*-1})$, leaving
\begin{align*}
		\sum_{i=1}^p\cF^*(\Phi_{f'}(\beta_i D))|\beta_i)
		=\lambda \sum_{i=1}^p  \Psi(D)(f'\otimes([\Z_p\times\Z_p^\times]|\beta_i^{*-1}))\\
		=\lambda\Psi(D)(f'_{\tau,\psi}\otimes([\Z_p\times\Z_p^\times]|U_p)).
\end{align*}
Applying Lemma \ref{L:U_p}, we see that this is equal to $\lambda\cF^*\Phi_{f'}(D)$. Thus $$\Phi_{f'}^k|U_p(D)=p\pi_k(\lambda\cF^*(\Phi_{f'}(D)))=p\lambda\Phi_{f'}^k(D).$$
\end{proof}

\subsection{Critical slope Eisenstein symbols}
In this section, we show that the modular symbols $R\Phi^{-k}_{f'}$ have the same $p$-adic $L$-functions as critical slope Eisenstein overconvergent modular symbols for good choices of $f'$.. 

\begin{lem}\label{L:satisfied}
The test functions $f'_{\tau,\psi}$, $\psi\neq 1$ and $f'_\ell$ satisfy the following conditions:
\begin{enumerate}
	\item[(H1)] $f'$ is good for the cusps $0,\frac{1}{p},\ldots,\frac{p-1}{p}$
	\item[(H2)] There exists $\lambda\in\C_p^\times$ with trivial $p$-adic valuation for which $f'|\beta_i^{*-1}=\lambda f'$ for all $i=1,\ldots,p$\end{enumerate}

\end{lem}
\begin{proof}
To see that $f'_\ell$ satisfies H1 we use the factorization $f'_\ell=f'_1\times f'_2$, where
\begin{equation*}
	f_1'=\bigotimes_{q\neq p} [\Z_q] \text{ and } f_2' =\left(\bigotimes_{q\neq p,\ell}[\Z_q]\right)\otimes([\Z_\ell]-\ell[\ell \Z_\ell]).
\end{equation*}
The test function $f'_2$ has Haar measure $h_\ell([\Z_\ell])-\ell h_\ell([\ell\Z_\ell])=1-\ell\frac{1}{\ell}=0$. Restricting $f'$ to lines parallel to ${0 \choose 1}$, we either get an identically zero test function, or the test function $f'_2$. In either case, the projection has Haar measure $0$, so $f'_\ell$ satisfies the vanishing hypothesis for the cusp $0$ and, since they are $\Gamma_0(\ell)$-equivalent to $0$, the cusps $\frac{1}{p},\ldots,\frac{p-1}{p}$.

To show the that $f'_{\tau,\psi}$ satisfies H1, let us fix $a\in\{1,\ldots,p-1\}$, $v={a \choose p}$,, $w= {b \choose d}$ and write $f=f'_{\tau,\psi}\otimes[\Z_p^2]$. Remark \ref{R:factorization} expresses $f$ as the product $\left(\sum_{i=1}^M\tau^{-1}(i)[i+L\Z]\right)\times \left(\sum_{j=1}^M \psi(j)[jL + N\Z]\right)$, and thus the $1$-dimensional test function $g(x)=f(v+wx)=f(ax+b,px+d)$ is equal to
\begin{align*}
	g(x)= \left(\sum_{i=1}^M\tau^{-1}(i)[i+L\Z](b+ax)\right)\times \left(\sum_{j=1}^M \psi(j)[jL + N\Z](px+d)\right)
\end{align*}
If $b\not\in\Z$, this is identically zero. Otherwise, $g(x)$ is periodic with respect to the lattice $apN\Z$, and since $(a,p)=1$, $g(x)$ is supported on $\Z$. Now $g(x)\neq 0$ implies there exist integers $i,j$ such that (i) $ax+b-i\equiv 0\pmod L$ and (ii) $px+d-jL\equiv 0\pmod{LM}$. Reducing (ii) modulo $L$, we see $px+d\equiv 0\pmod L$, which implies $i\equiv b-ap^{-1}d\pmod{L}$. As $x$ runs over the integers modulo $apLM$, every residue class $j$ modulo $M$ appears in the support of $g$. Therefore, the Haar measure of $g(x)$ is equal to
\begin{equation}
\frac{1}{apN}\sum_{x\in \Q/apN\Z} g(x)=\frac{1}{apN}\sum_{j\pmod{M}} \psi(j)=0
\end{equation}
since $\psi\neq1$. Therefore, $f'_{\tau,\psi}$ is good for the cusps $\frac{1}{p},\ldots,\frac{p-1}{p}$. The factorization of $f'_{\tau,\psi}$ shows that it is good for the cusp $0$, so $f'_{\tau,\psi}$ satisfies H1.

Finally, H2 follows from the straightforward computations $f'_{\tau,\psi}|\beta_i=\tau(p)f'_{\tau,\psi}$ (noting $\tau(p)$ is a unit, since the conductor of $\tau$ is not divisible by $p$) and $f'_\ell|\beta_i=f'_\ell$. 

\end{proof}

\begin{prop}\label{P:almost_eigen}
Let $k$ be an even integer, and suppose $f'$ satisfies H1 and H2. Then $R\Phi_{f'}^{-k}|U_p\{\infty,0\}=p^{k+1}\lambda R\Phi_{f'}^{-k}\{\infty,0\}$.
\end{prop}
\begin{proof}

Fix $g\in\cA^\dagger(\Z_p)$. The distribution $R\Phi_{f'}^{-k}|U_p\{\infty,0\}$ evaluated at $g$ is equal to
\begin{align*}
	R\Phi^{-k}_{f'}|U_p\{\infty,0\}(g)=\sum_{a=0}^{p-1} \langle \Phi_{f}'^{-k}\{\infty,\frac{a}{p}\}, \beta_a\cdot_k g\rangle_k
\end{align*}
The fact that $f'$ is good at the cusps $\frac{a}{p}$ means $\Phi_{f'}^{-k}\{\infty,\frac{a}{p}\}$ has no poles in the residue disk of $\infty$, except a possible simple pole at $\infty$. Lemma \ref{L:almost_equivariance} shows $\langle \Phi_{f}'^{-k}\{\infty,\frac{a}{p}\}, \beta_a\cdot_k g\rangle_k=p^k \langle \Phi_{f'}^{-k}\{\infty,\frac{a}{p}\}|_{-k}\beta_a, g\rangle_k$, so our sum is
\begin{align*}
	=\sum_{a=0}^{p-1} p^k \langle \Phi_{f'}^{-k}\{\infty,\frac{a}{p}\}|_{-k}\beta_a, g\rangle_k\\
	=\langle p^k\Phi_{f'}^{-k}|U_p\{\infty,0\},g\rangle_k\\
	= \langle p^{k} p\lambda \Phi_{f'}^{-k}\{\infty,0\},g\rangle_k\\
	=p^{k+1}\lambda R\Phi_{f'}^{-k}(g).
\end{align*}
We conclude that $R\Phi_{f'}^{-k}|U_p\{\infty,0\}=p^{k+1}\lambda R\Phi_{f'}^{-k}\{\infty,0\}$.
\end{proof}

\begin{prop}\label{P:same_L}
If $R\Phi_{f'}^{-k}|U_p\{\infty,0\}=p^{k+1}\lambda R\Phi_{f'}^{-k}\{\infty,0\}$ with $\lambda$ a $p$-adic unit, then $R\Phi_{f'}^{-k}$ differs from the critical slope modular symbol $\Phi^{crit}$ by a modular symbol with trivial $p$-adic $L$-function.
\end{prop}
\begin{proof}
The Banach space $\Symb_\Gamma(\mathbf{D}_k)$ decomposes as a direct sum 
\begin{equation*}
\Symb_{\Gamma}(\mathbf{D}_k)^{< k+1}\oplus\Symb_\Gamma(\mathbf{D}_k)^{k+1}\oplus\Symb_{\Gamma}(\mathbf{D}_k)^{>k+1 },
\end{equation*}
so we may decompose $R\Phi_{f'}^{-k}$ as the sum of modular symbols $\phi^{< k+1}$, $\phi^k$ and $\phi^{> k+1}$. We claim $L_p(\phi^{<k+1},s)$ and $L_p(\phi^{>k+1},s)$ are both identically zero. 

By Stevens' control theorem, $\phi^{<k+1}$ is classical, and since it is Eisenstein and non-critical slope, it must be ordinary. 
Applying Hida's ordinary projection $e$ to $R\Phi_{f'}^{-k}$, we get $(e\Phi_{f'}^{-k})\{\infty,0\}=\lim_{n\rightarrow \infty} p^{n(k+1)} \lambda^nR\Phi_{f'}^{-k}\{\infty,0\}=0$, so $\phi^{<k+1}\{\infty,0\}=0$. Similarly, we can conclude $\phi^{>k+1}\{\infty,0\}=0$ by iterating the operator $U_p/\lambda p^{k+1}$: $R\Phi_{f'}^{-k}\{\infty,0\}=\phi^{k+1}\{\infty,0\} + (\lim_{n\rightarrow\infty} (U_p/\lambda p^{k+1})^n \phi^{>k+1})\{\infty,0\}=\phi^{k+1}\{\infty,0\}+0$.

We conclude that the difference between $R\Phi_{f'}^{-k}$ and $\phi^{k+1}$ vanishes on $\{\infty,0\}$, and therefore has trivial $p$-adic $L$-function. 
\end{proof}

Now suppose $\tau,\psi$ are two primitive Dirichlet characters of prime-to-$p$ conductor $L, M$, respectively and that $\psi\neq 1$. Fixing an embedding $\overline{\Q}\hookrightarrow\C_p$, we view $\tau,\psi$ as characters valued in $\C_p^\times$. 

\begin{thm}
If $\psi\neq 1,\tau$ are primitive Dirichlet characters of prime-to-$p$ conductor, then $L_p(\Phi_{\tau,\psi}^{-k},s)=L_p(E_{k+2,\psi,\tau}^{crit},s)$ for all even $k>0$. If $k=0$ and $\psi,\tau$ have relatively prime conductor, then the same result holds.
\end{thm}
\begin{proof}
Bellaiche's theorem implies $\Symb_{\Gamma}^{\psi(-1)}(\bD_{k})[E_{k+2,\psi,\tau}^{crit}]$ is $1$-dimensional. Denote by $\Phi_{Eis}$ any non-zero element of $\Symb_{\Gamma}^{\psi(-1)}(\bD_{k})[E_{k+2,\psi,\tau}^{crit}]$. By Lemma \ref{L:satisfied}, we can apply Proposition \ref{P:almost_eigen} and Proposition \ref{P:same_L} to $\Phi_{\tau,\psi}^{-k}$ to conclude that, up to non-zero scalar, $L_p(\Phi_{Eis},s)=L_p(\Phi_{\tau,\psi}^{-k},s)$.
\end{proof}

In the following section, we compute $L_p(\Phi_{\tau,\psi}^{-k},s)$, thus giving formulas for $L_p(E_{k+2,\psi,\tau}^{crit},s)$ and $L_p(E_{2,\ell}^{crit},s)$.

\subsection{Computing $p$-adic $L$-functions}

\begin{lem}\label{L:restriction}
Suppose $\phi\in\Symb_{\Gamma}(\cD(\Z_p))$ satisfies $\phi|U_p\{\infty,0\}=\alpha\phi\{\infty,0\}$. Then
\begin{equation*}
	\phi\{\infty,0\} |_{\Z_p^\times} (z^n) = \left(1-\frac{p^n}{\alpha}\right)\phi\{\infty,0\} (z^n).
\end{equation*}
\end{lem}
\begin{proof}
This is a standard computation with $U_p$ and can be found, for example, in the proof of Proposition 6.3 in \cite{PolSte11}.
\end{proof}

We will also use the following trick.
\begin{lem}
Let $\mu\in\widetilde{\cD}(\Q_p)$. For all $s\in\Z_p\subset\cX_{\Q_p}$, $L_p(D_z\mu,s+1)=sL_p( \mu,s)$.
\end{lem}
\begin{proof}
Let $n$ be a non-negative integer. Then $L_p( D_z \mu,n+1)=\int_{\Z_p^\times} z^{n} d( D_z\mu)=\int_{\Z_p^\times} nz^{n-1}d\mu=nL_p(\mu,n)$. The result follows from the fact that the non-negative integers are dense in $\cX_{\Q_p}$.
\end{proof}

Now we come to the Main Theorem of this section:

\begin{thm}\label{T:padicLfunction}
Let $k\geq$ be a positive even integer. Suppose the test function $f'$ is factorizable as $f'=f_1'\times f_2'$, and that $f'$ is good for the cusps $0,\frac{1}{p},\ldots,\frac{p-1}{p}$.  Then the $p$-adic $L$-function of the modular symbol $R(\frac{1}{k}\Phi_{f'}^{-k})\in\Symb_{\Gamma}(\cD_k(\Z_p))$ is, up to a non-zero constant,
\begin{equation*}
	L_p\left(R\Phi_{f'}^{-k},s\right)={\wt{s}  \choose k+1}L_p(\xi_{f'_1},s-k)L_p(\xi_{f'_2},-s). 
\end{equation*}
\end{thm}
\begin{proof}

Denote by $\mu_{f'}\in\cD(\Z_p^2)$ the two-variable distribution satisfying $\Phi_{f'}\{\infty,0\}=\frac{1}{XY}\mu_{f'}$. Fixing $m\in\Z_{\geq 0}$, the moment $R\Phi_{f'}^{-k}\{\infty,0\}(z^m)$, is equal to the coefficient of $z^{-m+k}\frac{dz}{z}$ in 
\begin{equation*}
\Phi_{f'}^{-k}\{\infty,0\} = \sum_{n \geq 0} {-k \choose n} \left(\int_{\Z_p\times\Z_p^\times} (x/y)^ny^kd\mu_{f'}\right) z^{-n}\frac{dz}{z}.
\end{equation*}
Therefore $R\Phi_{f'}^{-k}(z^m)=0$ if $m<k$; otherwise
\begin{equation*}
	R\Phi_{f'}^{-k}\{\infty,0\}(z^m) = {-k \choose m-k} \int_{\Z_p\times\Z_p^\times}x^{m-k} y^{-m} d\mu_{f'}(x,y).
\end{equation*}
The binomial coefficient ${-k \choose m-k}$ is equal to $(-1)^{m-k} { m-1 \choose m-k}=(-1)^{m-k} {m-1 \choose k-1}$. Using the factorization of $\mu_{f'}$ gives
\begin{equation*}
\int_{\Z_p\times\Z_p^\times}x^{m-k} y^{-m} d\mu_{f'}(x,y)=\int_{\Z_p} x^{m-k} d(D_x\xi_{f'_1}) \int_{\Z_p^\times} y^{-m} d(D_y\xi_{f'_2}).
\end{equation*}
Note that the second factor $\int_{\Z_p^\times} y^{-m}d(D_y\xi_{f'_2})$ is equal to $L_p( D_y\xi_{f'_2},-m+1)=-mL_p(\xi_{f'_2},-m)$. Restricting to $\Z_p^\times$ and applying Lemma \ref{L:restriction}, a calculation shows  that $\int_{\Z_p^\times} x^m d(R{\Phi}_{f'}^{-k})=(-1)^{m-k}k {m \choose k+1} L_p(\xi_{f'_1},m-k)L_p(\xi_{f'_2},-m)$. Therefore 
\begin{equation}\label{E:payoff}
L_p(\frac{1}{k} R(\Phi_{f'}^k),m)=(-1)^{m-k}k {m \choose k+1} L_p(\xi_{f'_1},m-k)L_p(\xi_{f'_2},-m).
\end{equation}
Since the integers are dense in weight-space, we have the theorem in the case $k>0$.

%\begin{align*}
%	(-1)^{m-k+1}\left(1-\frac{p^{m-k-1}}{\alpha}\right) {m-1 \choose k-1} &\left(\int_{\Z_p}x^{m-k}d(D_x\xi_{f_1'})\right) mL_p(\xi_{f'_2},-m)\\.
%\end{align*}
%A computation shows $\left(1-\frac{p^{m-k-1}}{\alpha}\right)\int_{\Z_p} x^{m-k} d(D_x \xi_{f'_1})=\int_{\Z_p^\times} x^{m-k} d(D_x \xi_{f'_1})$, which is equal to 
%\begin{equation*}
%L_p( D_x\xi_{f'_1}, m-k+1)=(m-k)L_p(\xi_{f'_1},m-k).
%\end{equation*}
%Therefore, $\int_{\Z_p^\times} x^m d(R{\Phi}_{f'}^{-k})$ is equal to
%\begin{align*}
% (-1)^{m+1-k}{m-1 \choose k-1} (m-k)mL_p(\xi_{f'_1},m-k)L_p(\xi_{f'_2},-m)\\
% (-1)^{m-k} \frac{m(m-1)\cdots (m-k)}{(k-1)!} L_p(\xi_{f'},m-k)L_p(\xi_{f'_2},-m)\\
% =(-1)^{m-k}k {m \choose k+1} L_p(\xi_{f'_1},m-k)L_p(\xi_{f'_2},-m).
%\end{align*}

The case $k=0$ is similar; we briefly indicate the necessary changes while omitting the details. First, we use Lemma \ref{L:is_zero} to deduce
\begin{equation}
	\frac{1}{k}R(\Phi_{f'}^k)\{\infty,0\}= \sum_{n\geq 0} {k-1 \choose n} \left(\int_{\Z_p\times\Z_p^\times} (x/y)^n y^{k-1} d\mu_{f'}\right) z^{-n} \frac{dz}{z},
\end{equation}
and then use the factorization of $\mu_{f'}$ as $D_x\xi_{f'_1}\times \xi_{f'_2}$. Specializing to $k=0$ gives the result.
 \end{proof}

Now we are ready to prove the first case of Theorem C. 
\begin{proof}
Consider $f'_{\tau,\psi}$, with $\tau,\psi$ as above. The test function factorizes as $f_1'\times f_2'$, where $f'_1\otimes[\Z]=\sum_{i=1}^M\tau^{-1}(i)[i+L\Z]$ and $f_2'\otimes[\Z]=\sum_{j=1}^M \psi(j)[jL + N\Z]=\left(\sum_{j=1}^m\psi(j)[j+M\Z]\right)|[L]^{*-1}$. By Lemma \ref{L:satisfied}, we $f'_{\tau,\psi}$ satisfies the hypothesis of Theorem \ref{T:padicLfunction}. It follows from the results of \cite{Ste1} that $\xi_{f'_1}(z^n)$ is equal to the coefficient of $X^n$ in the Laurent series
\begin{align*}
\sum_{i \text{ mod }{L}} \tau^{-1}(i) \frac{e^{iX}}{1-e^{LX}},
\end{align*}
which is equal to $=B_{n,\tau^{-1}}/n=-L(\tau^{-1},1-n)$. It follows that
$$L_p(\xi_{f'_1},n)=\xi_{f'_1}|_{\Z_p^\times}(z^n)=-(1-\tau(p)p^{n-1})L(\tau^{-1},1-n)=L_p(\tau^{-1},n).$$
\begin{rem}
Note that our conventions here differ from other places in the literature, including \cite{BeDa}. In particular, we are writing $L_p(\tau^{-1},n)$ for the evaluation of $L_p(\tau,s)$ on the character $z\mapsto z^n$. 
\end{rem}
A similar calculation shows $L_p(\xi_{f_2'},s)=-L^{-s} L_p(\psi,s)$. Now if $(-1)^s=\psi(-1)$, then $L_p(E_{k+2,\psi,\tau}^{crit},s)=L_p(\frac{1}{k}R\Phi_{f'_{\tau,\psi}}^{-k},s)$ and
$$L_p(E_{k+2,\psi,\tau}^{crit},s)=L^{-s}{\wt{s}  \choose k+1}L_p(\tau^{-1},s-k)L_p(\psi,-s).$$
\end{proof}
%It follows from the results of \cite{Ste1} that $L_p(\xi_{f_1'},s)=L_p(\tau^{-1},s)$ and $L_p(\xi_{f_2'},s)=L^{1-s}L_p(\psi,s)$. Therefore,
%
%
%\begin{thm}
%Let $k\geq 0$ be an even integer and $\psi,\tau$ primitive Dirchlet characters with conductors $M,L$ both prime to $p$. If $k=0$ and $(M,L)=1$ or $k>0$ and $\psi\neq 1$, we have
%$$L_p(E_{k+2,\psi,\tau}^{crit},s)=L^{1-s}{\wt{s}  \choose k+1}L_p(\tau^{-1},s-k-1)L_p(\psi,1-s).$$
%\end{thm}
%
To prove the second case of Theorem C, we take $f'=f_\ell'$ and $k=0$. The same calculations show
$$L_p(E_{2,\ell}^{crit},s)=\wt(s)(1- \ell^{s})\zeta_p(s+1)  \zeta_p(1-s).$$
	%Scratch: (s-1)\zeta_p(1-(s-1))(1-\ell^{1-s})\zeta_p(1-(1-s))

\bibliographystyle{amsplain}
%\bibliography{PMSEEbib}
\providecommand{\bysame}{\leavevmode\hbox to3em{\hrulefill}\thinspace}
\providecommand{\MR}{\relax\ifhmode\unskip\space\fi MR }
\providecommand{\MRhref}[2]{%
  \href{http://www.ams.org/mathscinet-getitem?mr=#1}{#2}
}

\end{document}